\documentclass[11pt,letterpaper]{amsart}
\usepackage[usenames,dvipsnames]{color}
\usepackage{amsthm,amsfonts,amssymb,amsmath,amsxtra}
\usepackage[all]{xy}
\SelectTips{cm}{}
\usepackage{xr-hyper}
\usepackage[colorlinks=
   citecolor=Black,
   linkcolor=Red,
   urlcolor=Blue]{hyperref}
\usepackage{verbatim}
\usepackage{caption}

\usepackage{lineno}

\usepackage{mathrsfs}

\RequirePackage{xspace}
\RequirePackage{etoolbox}
\RequirePackage{varwidth}
\RequirePackage{enumitem}
\RequirePackage{tensor}
\RequirePackage{mathtools}
\RequirePackage{longtable}
\RequirePackage{multirow}

\def\ge{\geqslant}
\def\le{\leqslant}
\def\a{\alpha}
\def\b{\beta}
\def\g{\gamma}
\def\G{\Gamma}
\def\d{\delta}
\def\D{\Delta}
\def\L{\Lambda}
\def\e{\epsilon}

\def\o{\omega}

\def\ph{\phi}

\def\t{\tau}
\def\th{\theta}
\def\k{\kappa}
\def\l{\lambda}
\def\z{\zeta}

\def\i{^{-1}}

\def\<{\langle}
\def\>{\rangle}

\newcommand{\BA}{\ensuremath{\mathbb {A}}\xspace}

\newcommand{\BF}{\ensuremath{\mathbb {F}}\xspace}
\newcommand{{\BG}}{\ensuremath{\mathbb {G}}\xspace}

\newcommand{{\BK}}{\ensuremath{\mathbb {K}}\xspace}

\newcommand{\BQ}{\ensuremath{\mathbb {Q}}\xspace}
\newcommand{\BR}{\ensuremath{\mathbb {R}}\xspace}

\newcommand{\BW}{\ensuremath{\mathbb {W}}\xspace}

\newcommand{\BZ}{\ensuremath{\mathbb {Z}}\xspace}

\newcommand{\CA}{\ensuremath{\mathcal {A}}\xspace}
\newcommand{\CB}{\ensuremath{\mathcal {B}}\xspace}

\newcommand{\CD}{\ensuremath{\mathcal {D}}\xspace}

\newcommand{\CG}{\ensuremath{\mathcal {G}}\xspace}
\newcommand{\CH}{\ensuremath{\mathcal {H}}\xspace}
\newcommand{\CI}{\ensuremath{\mathcal {I}}\xspace}

\newcommand{\CL}{\ensuremath{\mathcal {L}}\xspace}

\newcommand{\CO}{\ensuremath{\mathcal {O}}\xspace}

\newcommand{\CR}{\ensuremath{\mathcal {R}}\xspace}

\newcommand{\CZ}{\ensuremath{\mathcal {Z}}\xspace}

\newcommand{\Ad}{{\mathrm{Ad}}}
\newcommand{\ad}{{\mathrm{ad}}}

\newcommand{\GL}{\mathrm{GL}}

\newcommand{\id}{\ensuremath{\mathrm{id}}\xspace}
\let\Im\relax
\DeclareMathOperator{\Im}{Im}

\newcommand{\red}{\ensuremath{\mathrm{red}}\xspace}

\DeclareMathOperator{\tr}{tr}

\newcommand{\ov}{\overline}

\def\brk{{\breve k}}
\def\dw{{\dot w}}

\def\COk{{\CO_{\brk}}}

\def\pr{{\rm pr}}

\def\ind{{\rm ind}}

\def\der{{\rm der}}

\def\stab{{\rm Stab}}
\def\ov{\overline}

\def\tx{{\textsl x}}
\def\sc{{\rm sc}}

\def\sfL{\textsf{L}}
\def\sfT{\textsf{T}}
\def\sfP{\textsf{P}}
\def\sfN{\textsf{N}}
\def\sfB{\textsf{B}}

\newtheorem{theorem}{Theorem}
\newtheorem{proposition}[theorem]{Proposition}
\newtheorem{lemma}[theorem]{Lemma}

\newtheorem{corollary}[theorem]{Corollary}

\theoremstyle{definition}

\newtheorem{remark}[theorem]{Remark}

\numberwithin{equation}{section}
\numberwithin{theorem}{section}

\setitemize[0]{leftmargin=*,itemsep=\the\smallskipamount}
\setenumerate[0]{leftmargin=*,itemsep=\the\smallskipamount}

\renewcommand{\to}{%
   \ifbool{@display}{\longrightarrow}{\rightarrow}%
   }
\let\shortmapsto\mapsto
\renewcommand{\mapsto}{%
   \ifbool{@display}{\longmapsto}{\shortmapsto}%
   }
\newlength{\olen}
\newlength{\ulen}
\newlength{\xlen}
\newcommand{\xra}[2][]{%
   \ifbool{@display}%
      {\settowidth{\olen}{$\overset{#2}{\longrightarrow}$}%
       \settowidth{\ulen}{$\underset{#1}{\longrightarrow}$}%
       \settowidth{\xlen}{$\xrightarrow[#1]{#2}$}%
       \ifdimgreater{\olen}{\xlen}%
          {\underset{#1}{\overset{#2}{\longrightarrow}}}%
          {\ifdimgreater{\ulen}{\xlen}%
             {\underset{#1}{\overset{#2}{\longrightarrow}}}
             {\xrightarrow[#1]{#2}}}}%
      {\xrightarrow[#1]{#2}}
   }
\makeatother
\newcommand{\xyra}[2][]{%
   \settowidth{\xlen}{$\xrightarrow[#1]{#2}$}%
   \ifbool{@display}%
      {\settowidth{\olen}{$\overset{#2}{\longrightarrow}$}%
       \settowidth{\ulen}{$\underset{#1}{\longrightarrow}$}%
       \ifdimgreater{\olen}{\xlen}%
          {\mathrel{\xymatrix@M=.12ex@C=3.2ex{\ar[r]^-{#2}_-{#1} &}}}%
          {\ifdimgreater{\ulen}{\xlen}%
             {\mathrel{\xymatrix@M=.12ex@C=3.2ex{\ar[r]^-{#2}_-{#1} &}}}
             {\mathrel{\xymatrix@M=.12ex@C=\the\xlen{\ar[r]^-{#2}_-{#1} &}}}}}%
      {\mathrel{\xymatrix@M=.12ex@C=\the\xlen{\ar[r]^-{#2}_-{#1} &}}}%
   }
\makeatletter
\newcommand{\xla}[2][]{%
   \ifbool{@display}%
      {\settowidth{\olen}{$\overset{#2}{\longleftarrow}$}%
       \settowidth{\ulen}{$\underset{#1}{\longleftarrow}$}%
       \settowidth{\xlen}{$\xleftarrow[#1]{#2}$}%
       \ifdimgreater{\olen}{\xlen}%
          {\underset{#1}{\overset{#2}{\longleftarrow}}}%
          {\ifdimgreater{\ulen}{\xlen}%
             {\underset{#1}{\overset{#2}{\longleftarrow}}}
             {\xleftarrow[#1]{#2}}}}%
      {\xleftarrow[#1]{#2}}
   }
\newcommand{\isoarrow}{%
   \ifbool{@display}{\overset{\sim}{\longrightarrow}}{\xrightarrow\sim}%
   }

\usepackage{upgreek}
\makeatletter
\newcommand{\colim@}[2]{%
  \vtop{\m@th\ialign{##\cr
    \hfil$#1\operator@font lim$\hfil\cr
    \noalign{\nointerlineskip\kern1.5\ex@}#2\cr
    \noalign{\nointerlineskip\kern-\ex@}\cr}}%
}
\newcommand{\colim}{%
  \mathop{\mathpalette\colim@{\rightarrowfill@\textstyle}}\nmlimits@
}
\newcommand{\prolim@}[2]{%
  \vtop{\m@th\ialign{##\cr
    \hfil$#1\operator@font lim$\hfil\cr
    \noalign{\nointerlineskip\kern1.5\ex@}#2\cr
    \noalign{\nointerlineskip\kern-\ex@}\cr}}%
}
\newcommand{\prolim}{%
  \mathop{\mathpalette\colim@{\leftarrowfill@\textstyle}}\nmlimits@
}
\makeatother

\begin{document}

\title{Decomposition of higher Deligne-Lusztig representations}
\author{Sian Nie}
\address{Academy of Mathematics and Systems Science, Chinese Academy of Sciences, Beijing 100190, China}

\address{ School of Mathematical Sciences, University of Chinese Academy of Sciences, Chinese Academy of Sciences, Beijing 100049, China}
\email{niesian@amss.ac.cn}


\begin{abstract}
Higher Deligne-Lusztig representations are virtual smooth representations of parahoric subgroups in a $p$-adic group. They are natural analogs of classical Deligne-Lusztig representations of reductive groups over finite fields. The most interesting higher Deligne-Lusztig representations are those attached to elliptic maximal tori, whose compact inductions are expected to realize supercuspidal representations of $p$-adic groups. Under a mild condition on $p$, in this paper we establish an explicit decomposition of these higher Deligne-Lusztig representations into irreducible summands. Surprisingly, all the irreducible summands are built in the same way as those in Yu's construction of irreducible supercuspidal representations, the only difference being that the Weil-Heisenberg representations in Yu's construction are replaced by their geometric analogs. As an application, we show that each irreducible supercuspidal representation of a $p$-adic group, attached to an unramified cuspidal datum, is a direct summand of the compact induction of a suitable higher Deligne-Lusztig representation, whenever the cardinality of the residue field of the $p$-adic field is not too small.
\end{abstract}

\maketitle

\section{Introduction} \label{sec-intro}

\subsection{Background and motivations} In the seminal work \cite{DL}, Deligne and Lusztig introduced a geometric way to construct representations of finite reductive groups, using the cohomology of so-called Deligne-Lusztig varieties. The (virtual) representations arisen this way are called Deligne-Lusztig representations, which play a central role in building up a powerful and elegant representation theory for finite reductive groups, known as Deligne-Lusztig theory.

Given the great success in the finite setting, in \cite{L79} Lusztig introduced natural extensions of the above constructions in the $p$-adic setting, which are referred to as higher Deligne-Lusztig varieties/representations. Since then, it has been a long-standing program to study these higher level analogs. The motivations are twofold. On one hand, higher Deligne-Lusztig varieties are of independent interest which admit very nice cohomological and arithmetic properties. On the other hand, the associated higher Deligne-Lusztig representations can be used to realize irreducible supercuspidal representations of $p$-adic groups, and has found many important applications in the local Langlands correspondence. We refer to \cite{L04}, \cite{Sta09}, \cite{B}, \cite{BW}, \cite{L04}, \cite{CS17}, \cite{Ch20}, \cite{CI21c}, \cite{CI21b}, \cite{CI21}, \cite{CI23}, \cite{CO}, \cite{I23a}, \cite{I23b},  \cite{DI}, \cite{CS23}, \cite{IN24a}, \cite{INT} for recent progress.

This paper is motivated by a question of Lusztig on the mysterious relation between higher Deligne-Lusztig representations and irreducible supercuspidal representations.
To describe it, we introduce some notation.

Let $k$ be a non-archimedean local field with a finite residue field $\BF_q$ of characteristic $p$ and of cardinality $q$. Denote by $\brk$ the completion of a maximal unramified extension of $k$. Let $F$ be the Frobenius automorphism of $\brk$ over $k$.

Let $G$ be a $k$-rational reductive group, and let $\tx$ be a point in its Bruhat-Tits building over $k$. Let $T \subseteq G$ be a $k$-rational and $\brk$-splitting maximal torus such that $\tx$ belongs to its apartment over $\brk$. Let $U$ be the unipotent radical of a $\brk$-rational Borel subgroup which contains $T$. Let $G_\tx$ and $T_\tx$ denote the corresponding  (connected) parahoric subgroups of $G = G(\brk)$ and $T = T(\brk)$ respectively.

Following \cite{L04} and \cite{CI21}, for any $r \in \BZ_{\ge 0}$ one can associate an $\ov\BF_q$-variety \[X_{T, U, r} = X_{T, U, \tx, r},\] which is called a higher Deligne-Lusztig variety. The $F$-fixed point group $G_\tx^F \times T_\tx^F$ acts on $X_{T, U, r}$ by the left/right multiplication, and hence acts on its $\ell$-adic cohomology groups $H_c^i(X_{T, U, r}, \ov \BQ_\ell)$, where $\ell \neq p$ is a different prime number. Thus, for any character $\phi: T^F \to \ov \BQ_\ell^\times$ of depth $r \ge 0$, the alternating sum of the corresponding isotropic subspaces \[R_{T, U, r}^G(\phi) := \sum_i (-1)^i H_c^i(X_{T, U, r}, \ov \BQ_\ell)[\phi |_{T_\tx^F}]\] gives a virtual representation of $G_\tx^F$, called a higher Deligne-Lusztig representation. When $r = 0$, $R_{T, U, r}^G(\phi)$ is the classical Deligne-Lusztig representation constructed in \cite{DL}.

The most basic and interesting higher Deligne-Lusztig representations $R_{T, U, r}^G(\phi)$ are those attached to elliptic maximal tori $T$, which we call elliptic higher Deligne-Lusztig representations. On one hand, any higher Deligne-Lusztig representation can be reduced to the study of elliptic ones. On the other hand, only elliptic higher Deligne-Lusztig representations could give rise to supercuspidal representations of the $p$-adic group $G^F = G(k)$. On the other hand, in \cite{Yu} Yu constructed a family of irreducible representations of (disconnected) parahoric subgroups in a purely algebraic way, whose compact inductions are irreducible  suprecuspical representations of $G^F$. A natural question, raised by Lusztig \cite{L79}, \cite{L04}, is that how to compare Yu's representations with elliptic higher Deligne-Lusztig representations.

\subsection{Irreducible decomposition} The main purpose of this paper is to give an explicit irreducible decomposition for elliptic higher Deligne-Lusztig representations, which presents a striking resemblance between their irreducible summands and Yu's representations. To this end, we employ a strategy inspired from the work of Chen-Stasinski \cite{CS23}. The key is to introduce a new $\ov\BF_q$-variety of Deligne-Lusztig type \[Z_{\phi, U, r},\] which also admits a natural action by $G_\tx^F \times T_\tx^F$. Similarly, using Deligne-Lusztig induction above one obtains another (virtual) $G_\tx^F$-module \[\CR_{T, U, r}^G(\phi) := \sum_i (-1)^i H_c^i(Z_{\phi, U, r}, \ov \BQ_\ell)[\phi|_{T_\tx^F}].\] The first step is to show that
\begin{proposition} \label{main-step}
    If $p$ is not a bad prime for $G$ and $p \nmid |\pi_1(G_\der)|$, then $R_{T, U, r}^G(\phi) = \CR_{T, U, r}^G(\phi)$ as virtual $G_\tx^F$-modules.
\end{proposition}
The proof of Proposition \ref{main-step} is to show the equalities \[\<R_{T, U, r}^G(\phi), R_{T, U, r}^G(\phi)\>_{G_r^F} = \<\CR_{T, U, r}^G(\phi), R_{T, U, r}^G(\phi)\>_{G_r^F} = \<\CR_{T, U, r}^G(\phi), \CR_{T, U, r}^G(\phi)\>_{G_r^F}.\] The first inner product is computed by Chan \cite{Ch24}. By extending methods of Lusztig \cite{L04}, Chen-Stasinski \cite{CS17} and Yu \cite{Yu}, we compute the last two and it turns out that all the three inner products coincide. This concludes the desired equality $R_{T, U, r}^G(\phi) = \CR_{T, U, r}^G(\phi)$.

\begin{remark} In a follow-up work \cite{IN24b}, we will give a cohomology-theoretic proof of the equality $R_{T, U, r}^G(\phi) = \CR_{T, U, r}^G(\phi)$, without computing inner products.\end{remark}

\smallskip

Given Proposition \ref{main-step}, the problem is reduced to the study of $\CR_{T, U, r}^G(\phi)$. Compared with $X_{T, U, r}$, the variety $Z_{\phi, U, r}$ has much simpler structure. As a result we are able to give an explicit decomposition of $\CR_{T, U, r}^G(\phi)$. To describe it, we invoke a result by Howe \cite{Howe} and Kaletha \cite{Kal} on Howe factorizations of smooth characters $\phi: T^F \to \ov \BQ_\ell^\times$. It says that if $p$ is as in Proposition \ref{main-step} then there is a generic datum (see \S\ref{subsec:Howe}) \[(G^i, \phi_i, r_i)_{0 \le i \le d},\] where $T \subseteq G^0 \subsetneq \cdots \subsetneq G^d = G$ are $k$-rational Levi subgroups of $G$ and $\phi_i: (G^i)^F \to \ov \BQ_\ell^\times$ are characters satisfying certain genericity conditions relative to the integer sequence $0 \le r_0 < r_1 < \cdots < r_{d-1} \le r_d$ such that \[\phi_{-1} := \phi\i \prod_{i=0}^d \phi |_{T^F}\] is a character (of $T^F$) of depth $0$. Following Yu's construction, we consider the following subgroups \begin{align*} K_\phi = K_{\phi, \tx} &= (G^0)_\tx (G^1)_\tx^{r_0/2} \cdots (G^d)_\tx^{r_d/2} \subseteq G_\tx \\ H_\phi = H_{\phi, \tx} &= (G^0)_\tx^{0+} (G^1)_\tx^{r_0/2} \cdots (G^d)_\tx^{r_d/2} \subseteq K_\phi,\end{align*} where $(G^{i+1})_\tx^s \subseteq (G^{i+1})_\tx$ denotes the $s$th Moy-Prasad subgroup attached to $\tx$ for $s \in \tilde \BR$. We show that each isotropic space \[H_c^i(Z_{\phi, U, r} \cap H_{\phi, r}, \ov \BQ_\ell)[\phi|_{(T_\tx^{0+})^F}]\] is a natural $K_\phi^F$-module, where $H_{\phi, r} \subseteq G_r$ is the natural image of $H_\phi$ and $T_\tx^{0+}$ is the pro-unipotent radical of $T_\tx$. Define  \[\k_\phi = \k_{\phi, U} := \sum_i (-1)^i H_c^i(Z_{\phi, U, r} \cap H_{\phi, r}, \ov \BQ_\ell)[\phi|_{(T_\tx^{0+})^F}].\] The next step is the following result.
\begin{proposition} \label{irr}
    The virtual $K_\phi^F$-module $\k_\phi$ is irreducible.
\end{proposition}

Note that the restriction $\k_\phi|_{H_{\phi, r}^F}$ is still obtained from the cohomological Deligne-Lusztig induction. Hence we can use the K\"{u}nneth formula (as in \cite[\S 6.6]{DL}) to show that \[\<\k_\phi|_{H_{\phi, r}^F}, \k_\phi|_{H_{\phi, r}^F}\>_{H_{\phi, r}^F} = 1,\] that is, $\k_\phi|_{H_\phi^F}$ is irreducible. However, the virtual $K_\phi^F$-module $\k_\phi$ is not constructed by Deligne-Lusztig induction. So it is much more challenging to show the irreducibility of $\k_\phi$. To achieve this, we prove the following remarkable concentration-at-one-degree property.
\begin{proposition} [Theorem \ref{one-degree}] \label{single-degree}
    There is a unique non-negative integer $n_\phi$ such that $H_c^i(Z_{\phi, U, r} \cap H_{\phi, r}, \ov \BQ_\ell)[\phi|_{(T_\tx^{0+})^F}] \neq \{0\}$ if and only if $i = n_\phi$.
\end{proposition}

\begin{remark}
    When the pair $(T, U)$ is of Coxeter type, an analogous concentration property also holds for certain closed subsets of higher Deligne-Lusztig varieties $X_{T, U, r}$, see \cite{B}, \cite{BW}, \cite{Ch20}, \cite{CI21b}, \cite{CI23} and \cite{IN24a}.
\end{remark}

Proposition \ref{single-degree} is proved by extending methods of \cite{BW} and \cite{IN24a}.

\smallskip

Having the above preparations, we now state the main theorem.
\begin{theorem} [Theorem \ref{decomposition}] \label{main}
    Let $p$ be as in Proposition \ref{main-step} and let notation be as above. Suppose $T$ is elliptic. Then
    \[R_{T, U, r}^G(\phi) = \CR_{T, U, r}^G(\phi) = \ind_{K_\phi^F}^{G_\tx^F} \k_\phi \otimes R_{T, U, 0}^{G^0}(\phi_{-1}) = \sum_\rho m_\rho  \ind_{K_\phi^F}^{G_\tx^F} \k_\phi \otimes \rho,\] where $R_{T, U, 0}^{G^0}(\phi_{-1})$ is a classical Deligne-Lusztig representation for the reductive quotient of $(G_\tx^0)^F$ (viewed as a $K_\phi^F$-module by inflation), and $\rho$ ranges over its irreducible summands with multiplicitiy $m_\rho$.

    Moreover, the summands $\ind_{K_\phi^F}^{G_\tx^F} \k_\phi \otimes \rho$ are pairwise non-isomorphic irreducible representations of $G_\tx^F$.
\end{theorem}

\begin{remark}
    Pioneering results were obtained by Chan-Oi \cite{CO} when $\phi$ is toral and by Chen-Stasinski \cite{CS17}, \cite{CS23} when $\phi$ is regular and $\tx$ is hyperspecial, using different methods. In both results,  $\pm R_{T, U, r}^G(\phi)$ is an irreducible $G_\tx^F$-module formulated in terms of Yu's representations \cite{Yu} and G\'{e}rardin's representations \cite{Ger}, respectively.
\end{remark}

\begin{remark} \label{rmk-compare}
    Note that in Theorem \ref{main} the irreducible $K_\phi^F$-modules \[\pm \k_\phi \otimes \rho\] are constructed in the same spirit of Yu's construction. The only different is that the Weil-Heisenberg representation $\k(\phi)$ used in Yu's construction is replaced by the representation $\k_\phi$ arising from geometry. In fact, the group $H_\tx^F$ has a finite quotient isomorphic to a Heisenberg $p$-group. Then the restrictions $\pm\k_\phi |_{H_\tx^F} \cong \k(\phi) |_{H_\tx^F}$ are inflated from the same Heisenberg representation determined by $\phi|_{(T_\tx^{0+})^F}$.
\end{remark}


\subsection{Application} Now we discuss an application of Theorem \ref{main} on supercuspidal representations of $p$-adic groups. In \cite{Yu}, Yu introduced the notion of cuspidal $G$-data and to each such datum $\Xi$ assigned an irreducible supercuspidal representation $\pi_\Xi$ of $G^F$. Thanks to work by Kim \cite{Kim} and Fintzen \cite{F21b}, it is known that when $p$ does not divide the order of the absolute Weyl group of $G$, all the irreducible supercuspidal representations of $G^F$ are exhausted by Yu's representations $\pi_\Xi$.

Recall that a cuspidal $G$-datum $\Xi$ contains a sequence of tamely ramified Levi subgroups $G^0 \subsetneq \cdots \subsetneq G^d = G$ and a point $\tx$ in the Bruhat-Tits building of $G^0$. We say $\Xi$ is unramified if $G^0$ (and hence all the Levi subgroups $G^i$) splits over $\brk$.

Our second main result is the following.
\begin{theorem} \label{consequence}
    Let $p$ be as in Theorem \ref{main} and assume that $q$ be sufficiently large. Let $\Xi$ be an unramified cuspidal $G$-datum as above. Then $\pi_\Xi$ is a direct summand of the compact induction \[\text{c-}\ind_{Z^F G_\tx^F}^{G^F} R_{T, U, r}^G(\phi),\] where $Z$ is the center of $G$ and $R_{T, U, r}^G(\phi)$ is some higher Deligne-Lusztig representation as in Theorem \ref{main}, extended to a $Z^F G_\tx^F$-module on which $Z^F$ acts via $\phi$.
\end{theorem}
We refer to Theorem \ref{sum} for the precise largeness condition on $q$. The proof is based on combining Theorem \ref{main} with the methods of Kim \cite{Kim} and Fintzen \cite{F21a}.

\begin{remark}
    If $\Xi = (S, \th)$ is a toral cuspidal $G$-datum, with $S$ an unramified elliptic maximal torus and $\th: S^F \to \BQ_\ell^\times$ a character of depth $\ge r$, it is proved by Chan and Oi \cite{CO} that \[\pi_\Xi \cong \text{c-}\ind_{Z^F G_\tx^F}^{G^F} R_{S, V, r}^G(\th \varepsilon[\th]),\] where $\tx$ is some/any $F$-fixed point in the apartment of $S$ over $\brk$, and $\varepsilon[\th]$ is certain quadratic character of $S^F$. For general $\Xi$, there is a lack of information on the construction of $R_{T, U, r}^G(\phi)$ in Corollary \ref{consequence}. The reason is that the relation between the Weil-Heisenberg representations $\k(\phi)$ and their geometric analogs $\k_\phi$ is unclear. Guided by the work \cite{CO} and \cite{FKS}, we expect that $\pm \k_\phi = \e|_{K_\phi^F} \k(\phi)$, where $\e$ is the quadratic character defined in \cite[Theorem 4.1.13]{FKS}.
\end{remark}

\subsection{Structure of the paper} The paper is organized as follows. In \S \ref{sec-R} we recall the inner product formula and the degeneracy property of higher Deligne-Lusztig representations due to Chan \cite{Ch24}, which will play an essential role in our computation. In \S \ref{sec-CR}, we introduce the variety $Z_{\phi, U, r}$ and the associated representation $\CR_{T, U, r}^G(\phi)$. In \S \ref{sec-product}, we compute the inner product between $\CR_{T, U, r}^G(\phi)$ and $R_{T, U, r}^G(\phi)$. This is achieved by extending methods from \cite{L04}, \cite{CS17} and \cite{Ch24}. In \S \ref{sec-algebraic}, we compute the self inner product of $\CR_{T, U, r}^G(\phi)$, which completes the proof of the equality $\CR_{T, U, r}^G(\phi) = R_{T, U, r}^G(\phi)$. \S \ref{sec: concentration} is devoted to the proof of Proposition \ref{single-degree}. In \S \ref{sec-decomp} we decompose $\CR_{T, U, r}^G(\phi)$ in to irreducible representations of Yu's type and finishes the proof of Theorem \ref{main}. In the last section, we prove Theorem \ref{consequence}.

\subsection*{Acknowledgement}
We would like to thank Zhe Chen for explaining ideas in his joint work \cite{CS17}, \cite{CS23} with Alexander Stasinski, which inspired the construction of $Z_{\phi, U, r}$ and the equality $R_{T, U, r}^G(\phi) = \CR_{T, U, r}^G(\phi)$. It is also clear from the context that the of Proposition \ref{main-step} depends heavily on results and methods by Chan \cite{Ch24}. We are also grateful to Alexander Ivanov for the collaboration on higher Deligne-Lusztig varieties and representations, which inspired the proof of Proposition \ref{single-degree}. Finally, we thank George Lusztig for helpful comments which improve the exposition of this paper significantly.

\subsection*{Conventions and notation}
Let $k$ be a non-archemedean field with a finite residue field $\BF_q$ of cardinality $q$ and of characteristic $p \neq 2$. Let $\brk$ be the completion of a maximal unramified extension of $k$. Denote by $\CO_k$ and $\COk$ the integer rings of $k$ and $\brk$ respectively. Fix a uniformizer $\varpi \in \CO_k$. Let $F$ be the Frobenius automorphism of $\brk$ over $k$.

Let $G$ be a connected $k$-rational reductive group splitting over $\brk$. We write $Z(G)$ for the center of $G$, $G_\der$ for the derived subgroup $G$ and $G_\sc$ for the simply connected covering of $G_\der$. Let $\CB(G, k)$ denote the (enlarged) Bruhat-Tits building of $G$ over $k$. By the Bruhat-Tits building theory, to each point $\tx \in \CB(G, k)$ one can associate a connected parahoric $\CO_k$-model $\CG_\tx$ of $G$, together with a filtration of Moy-Prasad subgroups $\CG_\tx^r$ for $r \in \widetilde\BR_{\ge 0}$. Here $\widetilde\BR = \BR \sqcup \{r+; r \in \BR\}$ with the usual order given by $s < s+ < r$ for any $s < r \in \BR$. For $s \le r \in \widetilde\BR_{\ge 0}$ we denote by $G_r^s$ the $\BF_q$-rational smooth affine group scheme, which represents the perfection of the functor \[R \mapsto \CG_\tx^s(\BW(R)) / \CG_\tx^{r+}(\BW(R)),\] where $R$ is a $\BF_q$-algebras, and $\BW(R)$ is the Witt ring of $R$ if ${\rm char}~k = 0$ and $\BW(R) = R[[\varpi]]$ otherwise. Let $H \subseteq G$ be a closed $\brk$-rational subgroup. We denote by $H_r^s \subseteq G_r^s$ the closed subgroup defined in \cite[\S 2.6]{CI21}. If $H$ is $k$-rational, the Frobenius $F$ acts on $H_r$ in a natural way. By abuse of notation, we will write $H = H(\brk)$, $G_\tx^s = \CG_\tx^s(\COk) \subseteq G$ and $H_r = H_r(\ov\BF_q)$. In particular, $G^F = G(k)$ and $G_r^F = G_r(\BF_q) = G_\tx^F / (G_\tx^{r+})^F$.

Let $T \subseteq G$ be a $k$-rational maximal torus splitting over $\brk$. We denote by $\CA(T, \brk)$ the apartment of $T$ inside the Bruhat-Tits building $\CB(G, \brk)$ of $G$ over $\brk$. Write $\Phi(G, T)$ for the root system of $T$ in $G$ over $\brk$.

All representations of groups in this paper have coefficients in $\ov\BQ_\ell$, where $\ell \neq p$ is a different prime number. Let $A$ be a group. For two subgroups $A_1, A_2 \subseteq A$ let $[A_1, A_2]$ denote the subgroup generated by the commutators $[a_1, a_2] := a_1 a_2 a_1\i a_2\i$ for all $a_1 \in A_1$ and $a_2 \in A_2$. For $h \in A$, $K$ a subgroup of $A$, and $\rho$ a representation $K$, we write ${}^h K = h K h\i$ and ${}^h \rho$ the representation of ${}^h K$ such that ${}^h \rho(x) = \rho(h\i x h)$ for $x \in {}^h K$. We say $h$ intertwines $\rho$ if $\hom_{K \cap {}^h K}(\rho, {}^h \rho)$ is non-trivial. Suppose the group $A$ acts on a set $Y$. We denote by $Y^A \subseteq Y$ the set of elements fixed by $A$, and by $\stab_A(y)$ the stabilizer of $y \in Y$ in $A$.

Let $X$ be a $\ov\BF_q$-variety. For $i \in \BZ$ we denote by $H_c^i(X, \ov\BQ_\ell)$ the $i$th $\ell$-adic cohomology space of $X$ with compact support. Suppose that $X$ admits an algebraic action by the product of two finite groups $A_1$ and $A_2$. Then $H_c^i(X, \ov\BQ_\ell)$ is a representation of $A_1 \times A_2$ by functoriality. For a character $\th$ of $A_2$, we write $H_c^i(X, \ov\BQ_\ell)[\chi] \subseteq H_c^i(X, \ov\BQ_\ell)$ for the $\th$-isotropic subspace, which is a representation of $A_1$ in the natural way. We write \[H_c^*(X, \ov\BQ_\ell)[\th] = \sum_{i \in \BZ} (-1)^i H_c^i(X, \ov\BQ_\ell)[\th],\] which is a virtual representation of $A_1$.

\section{Higher Deligne-Lusztig representations}\label{sec-R}
Let $G$ be a connected $k$-rational reductive group which splits over $\brk$. Throughout out the paper, we make the following assumption \[\tag{*} \text{$p$ is not a bad prime for $G$ and does not divide $|\pi_1(G_\der)|$.}\] Moreover, we fix a point $\tx \in \CB(G, k)$ except the last section.


\subsection{The representations $R_{T, U, r}^G(\phi)$} Let $T \subseteq G$ be a $k$-rational and $\brk$-splitting maximal torus such that $\tx \in \CA(T, \brk)$. Let $B = T U \subseteq G$ be a Borel subgroup with $U$ the unipotent radical. Let $r \in \BR_{\ge 0}$. The associated higher/parahoric Deligne-Lusztig variety is defined by \[X_{T, U, r} = X_{G, T, U, r}  = \{g \in G_r; g\i F(g) \in F U_r\}.\] There is a natural action of $G_r^F \times T_r^F$ on $X_{T, U, r}$ given by $(g, t): x \mapsto g x t$.

Let $\phi: T^F \to \ov \BQ_\ell^\times$ be a smooth character. The depth $\phi$, denoted by $r_\phi$, is the least non-negative integer $s \in \BZ_{\ge 0}$ such that $\phi$ is trivial over $(T_\tx^{s+})^F$. Suppose that $r_\phi \le r$. Then $\phi$ can be viewed as a character of $T^F / (T_\tx^{r+})^F$. The attached higher Deligne-Lusztig representation is defined by \[R_{T, U, r}^G(\phi) = H_c^*(X_{T, U, r}, \ov \BQ_\ell)[\phi|_{T_r^F}],\] which is a virtual representation of $G_r^F$. Using the natural projections $G_\tx^F \to G_s^F \to G_r^F$ for $s \ge r$, we also view $H_c^i(X_{T, U, r}, \ov \BQ_\ell)[\phi|_{T_r^F}]$ and $R_{T, U, r}^G(\phi)$ as representations of $G_s^F$ and $G_\tx^F$.

\subsection{Properties} We recall several important properties on the representations $R_{T, U, r_\phi}^G(\phi)$ established by Chan \cite{Ch24}. These results will play an essential role in the paper. The first result implies that the $G_\tx^F$-module $R_{T, U, r}^G(\phi)$ is in independent of the choice $r \ge r_\phi$ when $T$ is elliptic.
\begin{theorem} \cite[Theorem 5.2]{Ch24} \label{degenracy}
    Assume $T$ is elliptic. Then there is an integer $m$ such that for any $i \in \BZ$ we have \[H_c^i(X_{T, U, r}, \ov \BQ_\ell)[\phi|_{T_r^F}] \cong H_c^{i + 2m} (X_{T, U, r_\phi}, \ov \BQ_\ell)[\phi|_{T_{r_\phi}^F}]\] as $G_r^F$-modules. In particular, $R_{T, U, r}^G(\phi) \cong R_{T, U, r_\phi}^G(\phi)$ as $G_r^F$-modules.
\end{theorem}

The next result is a projection formula for $R_{T, U, r}^G(\phi)$.
\begin{proposition} \cite[Proposition 3.7]{Ch24} \label{ass-R}
    Let $\th$ be a  character of $G^F$ which is trivial over $G_\sc^F$ and $(G_\tx^{r+})^F$. Then $R_{T, U, r}^G(\phi) \otimes \th|_{G_r^F} \cong R_{T, U, r}^G(\phi \otimes \th|_{T_r^F})$ as $G_r^F$-modules.
\end{proposition}

Let $W_{G_r}(T_r) = (N_T)_r / T_r$, where $N_T$ denotes the normalizer of $T$ in $G$. Then $W_{G_r}(T_r)^F$ permutes characters of $T_r^F$ in a natural way. The last result is a inner product formula for $R_{T,U,r}^G(\phi)$ with $T$ elliptic.
\begin{theorem} \cite[Theorem 6.2]{Ch24} \label{R-R}
    Assume that $T$ is elliptic. Then \[\<R_{T,U,r}^G(\phi), R_{T,U,r}^G(\phi)\>_{G_r^F} = |\stab_{W_{G_r}(T_r)^F}(\phi |_{(T_r)^F})|.\] Moreover, $R_{T,U,r}^G(\phi)$ is independent of the choice of $B = T U$ containing $T$.
\end{theorem}

\begin{remark}
    In fact, Chan proved a much stronger version of the inner product formula. We refer to loc. cit. for the precise statement.

    If $\phi$ is generic, the inner product formula in Theorem \ref{R-R} is proved in \cite{L04}, \cite{Sta09} and \cite{CI21} without the elliptic assumption. If $T$ is of Coxeter type, it is proved by \cite{DI} and \cite{INT} when $q$ is not too small.
\end{remark}


\section{A new class of representations} \label{sec-CR}

In this section, we introduce the main geometric objects $Z_{\phi, U, r}$ and their cohomological induced representations $\CR_{T, U, r}^G(\phi)$ attached to smooth characters $\phi$ of $T^F$. We will show that $\CR_{T, U, r}^G(\phi)$ behaves similarly as $R_{T, U, r}^G(\phi)$ introduced in \S \ref{sec-R}.


\subsection{Generic datum} \label{subsec:generic}

A generic datum of $G$ is a tuple $\L = (G^i, \phi_i, r_i)_{0 \le i \le d}$ such that
\begin{itemize}
    \item $G^0 \subsetneq G^1 \subsetneq \cdots \subsetneq G^d = G$ are unramified Levi subgroups of $G$;

    \item $\phi_i$ is a character of $(G^i)^F$ for $0 \le i \le d$;

    \item $0 =: r_{-1} < r_0 < \cdots < r_{d-1} \le r_d$ if $d \ge 1$ and $0 \le r_0$ if $d = 0$;

    \item $\phi_i$ is of depth $r_i$ and is $(G^i, G^{i+1})$-generic in the sense of \cite[\S 9]{Yu} for $0 \le i \le d-1$;

    \item $\phi_d$ is of depth $r_d$ if $r_{d-1} < r_d$ and is trivial otherwise.

\end{itemize}
Moreover, we say $\L$ is normalized if (the pull-back of) $\phi_i$ is trivial over $(G_\sc^i)^F$ for $0 \le i \le d$.

\smallskip

Let $\L = (G^i, \phi_i, r_i)_{0 \le i \le d}$ be a normalized generic datum such that $\tx \in \CB(G^0, k)$. We define
the following $F$-stable subgroups
\begin{align*}
K_\L = K_{G, \L} &= (G^0)_\tx (G^1)_\tx^{r_0/2} \cdots (G^d)_\tx^{r_{d-1}/2}; \\
H_\L = H_{G, \L} &= (G^0)_\tx^{0+} (G^1)_\tx^{r_0/2} \cdots (G^d)_\tx^{r_{d-1}/2}; \\ K_\L^+ = K_{G, \L}^+ &= (G^0)_\tx^{0+} (G^1)_\tx^{r_0/2+} \cdots (G^d)_\tx^{r_{d-1}/2+}; \\ E_\L = E_{G, \L} &= (G^0_\der)_\tx^{0+,0+} (G^1_\der)_\tx^{r_0+, r_0/2+} \cdots (G^d_\der)_\tx^{r_{d-1}+, r_{d-1}/2+}. \end{align*} Here $(G^i_\der)_\tx^{r_{i-1}+, r_{i-1}/2+} \subseteq G_\tx$ is the subgroup generated by $(G^i_\der)_\tx^{r_{i-1}+}$ and $(G^\a)_\tx^{r_{i-1}/2+}$ for $\a \in \Phi(G^i, S) \setminus \Phi(G^{i-1}, S)$, where $S$ is any $\brk$-splitting maximal torus of $G^0$, and $G^\a \subseteq G$ is the root subgroup corresponding to $\a$.

For $0 \le i \le d$ let $\hat \phi_i$ be the character of $(G^i)_\tx^F (G_\tx^{r_i/2+})^F$ defined in \cite[\S 4]{Yu}, which extends $\phi_i$. We define a character of $(K_\L^+)^F$ by \[\chi_\L = \prod_{i = 0}^d \hat \phi_i |_{(K_\L^+)^F}.\] By definition $\chi_\L$ is trivial on $(E_\L)^F$.

\smallskip

Let $B \supseteq T$ be $\brk$-rational Borel subgroup of $G$ with unipotent radical $U$. Note that the set $\Phi(B, T)$ of roots in $B$ forms a positive system of $\Phi(G, T)$. We say a $\brk$-rational Levi subgroup $M \supseteq T$ is standard with respect to $B$ or $U$ if $\Phi(M, T)$ is standard with respect to $\Phi(B, T)$, that is, simple roots of $\Phi(M \cap B, T) = \Phi(M, T) \cap \Phi(B, T)$ are also simple roots of $\Phi(B, T)$. In this case, $M$ and $U$ generates a parabolic subgroup $P = M N$ of $G$, where $N \subseteq U$ is the unipotent radical of $P$.
\begin{lemma} \label{standard-U}
    There exists a $\brk$-rational Borel subgroup $B$ containing $T$ such that $(G^i)_{0 \le i \le d}$ is standard with respect to $B$.
\end{lemma}
\begin{proof}
    We fix a Borel subgroup $B$ containing $T$. Let $W$ be the Weyl group of the root system $\Phi(G, T)$. It suffices to show there exists $w \in W$ such that $\Phi(G^i, T)$ is standard with respect to $w(\Phi(B, T))$ for $0 \le i \le d$.

    We argue by induction on $d$. If $d = 1$, the statement is a well-known result on root systems. Now assume that $d = n \ge 2$ and the statement holds for $d \le n-1$. We show it also holds for $d = n$. Indeed, by induction hypothesis there exists $x \in W$ such that $\Phi(G^d, T)$ is standard with respect to $x(\Phi(B, T))$. Note that $\Phi(G^d, T) \cap x(\Phi(B, T))$ is a positive system of $\Phi(G^d, T)$. By induction hypothesis there exists $u \in W(G^d)$ such that $\Phi(G^i, T)$ ($0 \le i \le d-1$) is standard with respect to \[u(\Phi(G^d, T) \cap x(\Phi(B, T))) = \Phi(G^d, T) \cap ux(\Phi(B, T)).\] Here $W(G^d)$ denotes the Weyl group of $\Phi(G^d, T)$. Let $w = u x$. By construction, it follows that $\Phi(G^i, T)$ is standard with respect to $w(\Phi(B, T))$ for $0 \le i \le d$. The proof is finished.
\end{proof}

\subsection{Howe factorization}\label{subsec:Howe}
Let $T$ be a $k$-rational and $\brk$-splitting maximal torus of $G$ such that $\tx \in \CA(T, \brk)$. Let $\phi$ be a character of $T^F$ of depth $r_\phi \ge 0$. Following \cite{Kal}, a Howe factorization of $\phi$ is a pair $(\L, \phi_{-1}) = (G^i, \phi_i, r_i)_{-1 \le i \le d}$, where $\phi_{-1}$ is a character of $T^F$ of depth $r_{-1} := 0$ and $\L = (G^i, \phi_i, r_i)_{0 \le i \le d}$ is a normalized generic datum such that $T = G^{-1} \subseteq G^0$ and \[\phi = \prod_{i=-1}^d \phi_i|_{T^F}.\]

The following existence result is proved by Kaletha \cite[Theorem 3.6.7]{Kal}, under the assumption (*).
\begin{theorem} \label{existence-Howe}
    Each character $\phi$ of $T^F$ has a Howe factorization.
\end{theorem}
Let $(\L, \phi_{-1}) = (G^i, \phi_i, r_i)_{-1 \le i \le d}$ be a Howe factorization of $\phi$. We put $K_\phi = K_{G, \phi} = K_{G, \L}$, $H_\phi = H_{G, \phi} = H_{G, \L}$, $\phi^\natural = \chi_\L$ and so on, which are independent of the choices of $(\L, \phi_{-1})$. Note that $K_\phi^+ = E_\phi T_\tx^{0+}$ and hence $(K_\phi^+)^F = E_\phi^F (T^{0+})^F$. As $\phi$ is trivial over $E_\phi^F \cap T^F$, it follows that $\phi^\natural$ is the unique extension of $\phi$ which is trivial on $E_\phi^F$.

Let $\g \in \Phi(G, T)$. We denote by $T^\g$ the one-dimensional subtorus of $T$ corresponding the coroot of $\g$. Define $i(\g) = i^\phi(\g)$ to be the integer $0 \le i \le d$ such that $\g \in \Phi(G^i, T) \setminus \Phi(G^{i-1}, T)$, and define $r(\g) = r^\phi(\g) = r_{i^\phi(\g)-1}$.
\begin{lemma} \label{facts}
    We have the following properties.
    \begin{itemize}
        \item The subgroups $[K_{\phi, r}, K_{\phi, r}^+] \subseteq E_{\phi, r} \subseteq K_{\phi, r}^+$ are normalized by $K_{\phi, r}$;

        \item The natural multiplication map induces an isomorphism \[h: (E_{\phi, r} \cap T_r) \times \prod_{\g \in \Phi(G, T)} (G^\g)_r^{r(\g)/2+} \overset \sim \longrightarrow E_{\phi, r},\] where $E_{\phi, r} \cap T_r$ is generated by $(T^\g)_r^{r(\g)+}$ for $\a \in $;

        \item The intersection $E_{\phi, r} \cap T_r)$ is an affine space, and hence so is $E_{\phi, r}$.
    \end{itemize}
\end{lemma}
For $0 \le i \le j \le d$ we have \[[(G^i)_r^{r_{i-1}/2}, (G^j)_r^{r_{j-1}/2+}], [(G^i)_r^{r_{i-1}/2+}, (G^j)_r^{r_{j-1}/2}] \subseteq (G^j_\der)_r^{r_{j-1}+, r_{j-1}/2+} \subseteq (G^j_\der)_r^{r_{j-1}/2+}.\] Hence the first statement follows.

By construction, $E_{\phi, r}$ is generated by the image $\Im h$ of $h$. Moreover, one checks that $\Im h$ is a subgroup of $G_r^{0+}$ (since the set affine roots appearing in $E_{\phi, r}$ is closed under addition). Hence $h$ is surjective. By the Iwahori decomposition, the natural multiplication map \[T_r^{0+} \times \prod_{\g \in \Phi(G, T)} (G^\g)_r^{0+} \overset \sim \longrightarrow G_r^{0+}\] is an isomorphism. Hence $h$ is injective, and the second statement follows.

Let $\pi: G_\sc \to G_\der$ be the simply connected covering. Let $T_\der = G_\der \cap T$ and $T_\sc = \pi\i(T_\der)$. By Lemma \ref{standard-U}, there exists a base $\D$ of $\Phi(G, T)$ and a sequence of subsets $J^1 \subsetneq \cdots \subsetneq J^d = \D$ such that $J^i$ is a base of $\Phi(G^i, T)$ for $0 \le i \le d$. Then we have an injective homomorphism \[f: \prod_{\a \in \D} (T^\a)_r^{r(\a)+} \to T_\sc^{0+}.\] By \cite[Lemma 3.3.2]{Kal}, the restriction of $\pi$ to $T_\sc$ induces an isomorphism $\pi_T^+: T_\sc^{0+} \overset \sim \longrightarrow T_\der^{0+}$. Hence composition $\pi_T^+ \circ f$ is injective. Moreover, by construction we have $\Im (\pi_T^+ \circ f) = E_{\phi, r}$. Then the last statement follows by noticing that each group $(T^\a)_r^{r(\a)+}$ for $\a \in \D$ is an affine space.
\smallskip

\subsection{The representations $\CR_{T, U, r}^G(\phi)$} \label{subsec:CR} Let $T$ and $\phi$ be as in \S \ref{subsec:Howe}. Let $B = TU$ and $\ov B = T \ov U \subseteq G$ be two opposite Borel subgroups containing $T$, where $U$ and $\ov U$ are their unipotent radicals respectively. We define the following Iwahori-type subgroup \[\CI_{\phi, U} = \CI_{G, \phi, U} = (K_\phi \cap U) (E_\phi \cap T) (K_\phi^+ \cap \overline U) \subseteq K_\phi.\] Let $r \ge r_\phi$. Let $K_{\phi, r}$ be the image of $K_\phi$ under the quotient map $G_\tx \to G_r$, and define $H_{\phi ,r}$, $E_{\phi, r}$, $K_{\phi, r}^+$ and $\CI_{\phi, U, r}$ in a similar way.

We consider the variety \[Z_{\phi, U, r} = Z_{G, \phi, U, r} = \{g \in G_r; g\i F(g) \in F\CI_{\phi, U, r}\},\] which admits a natural action of $G_r^F \times T_r^F$ by left/right multiplication. The isotropic subspace \[\CR_{T, U, r}^G(\phi) :=  H_c^*(Z_{\phi, U, r}, \ov \BQ_\ell)[\phi|_{T_r^F}]\] gives a virtual representation of $G_r^F$. We put \begin{align*} Z_{\phi, U, r}^K &= Z_{\phi, U, r} \cap K_{\phi, r}; \\ \quad Z_{\phi, U, r}^H &= Z_{\phi, U, r} \cap H_{\phi, r}; \\ Z_{\phi, U, r}^L &= Z_{\phi, U, r} \cap L_r, \end{align*} where $T \subseteq L$ is a Levi subgroup of $G$. Note that $Z_{\phi, U, r}^K$ admits a natural action of $K_{\phi, r}^F \times T_r^F$ by left/right multiplication. Hence for each $i \in \BZ$, the isotropic subspace \[H_c^i(Z_{\phi, U, r}^K, \ov\BQ_\ell)[\phi |_{T_r^F}]\] gives a representation of $K_{\phi, r}^F$. As $\CI_{\phi, U, r} \subseteq K_{\phi, r}$, it follows that \[Z_{\phi, U, r} = \sqcup_{\g \in G_r^F / K_{\phi, r}^F} \g Z_{\phi, U, r}^K.\] In particular, we have $H_c^i(Z_{\phi, U, r}, \ov\BQ_\ell)[\phi |_{T_r^F}] \cong \ind_{K_{\phi, r}^F}^{G_r^F} H_c^i(Z_{\phi, U, r}^K, \ov\BQ_\ell)[\phi|_{T_r^F}]$ as $G_r^F$-modules.

\begin{proposition} \label{natural}
    We have $Z_{\phi, U, r}^K = Z_{\phi, U, r}^K E_{\phi, r}$. For each $i \in \BZ$ the quotient map $Z_{\phi, U, r}^K \to Z_{\phi, U, r}^K / E_{\phi, r}$ induces an isomorphism of $G_r^F$-modules \[H_c^{i + 2\dim E_{\phi, r}}(Z_{\phi, U, r}^K, \ov \BQ_\ell)[\phi|_{T_r^F}] \cong H_c^i(Z_{\phi, U, r}^K / E_{\phi, r}, \ov \BQ_\ell)[\phi|_{T_r^F}].\]

    Moreover, via left/right multiplication, $(K_{\phi, r}^+)^F$ acts on $H_c^i(Z_{\phi, U, r}^K, \ov \BQ_\ell)[\phi|_{T_r^F}]$ by the character $\phi^\natural$.
\end{proposition}
\begin{proof}
    The first statement follows from the inclusion $E_{\phi, r} \subseteq \CI_{\phi, U, r}$. The second follows from that $E_{\phi, r}$ is isomorphic to an affine spaces, see Lemma \ref{facts}. In particular, the action of $E_{\phi, r}^F$ on $ H_c^i(Z_{\phi, U, r}^K, \ov \BQ_\ell)$, induced by left/right multiplication, is trivial.  Let $h \in (K_{\phi, r}^+)^F$. As $(K_{\phi, r}^+)^F = (T_r^{0+})^F E_{\phi, r}^F$ we have $h = t h_1$ for some $t \in (T_r^{0+})^F$ and $h_1 \in E_{\phi, r}^F$. Note that $[(K_{\phi, r}^+)^F, K_{\phi, r}^F] \subseteq E_{\phi, r}^F$. For any $g \in Z_{\phi, U, r}$ we have \[ h g E_{\phi, r} = g E_{\phi, r} h = g h E_{\phi, r} = g t E_{\phi, r} = g E_{\phi, r} t,\] where the second equality from that $[K_{\phi, r}, K_{\phi, r}^+] \subseteq E_{\phi, r}$. Thus the left/right action of $h$ on $H_c^i(Z_{\phi, U, r}^K, \ov \BQ_\ell)[\phi|_{T_r^F}]$ is given the scalar $\phi(t) = \phi^\natural(h)$ as desired.
\end{proof}

\begin{corollary} \label{restriction}
    Assume $r_\phi > 0$. Let $\rho$ be an irreducible $G_r^F$-module which appears in $\CR_{T, U, r}^G(\phi)$. Then $\rho|_{(G_r^{r_\phi})^F}$ is a sum of weight spaces on which $(G_r^{r_\phi})^F$ acts via the characters ${}^\g \phi^\natural$ for $\g \in G_r^F / K_{\phi, r}^F$.
\end{corollary}
\begin{proof}
    By assumption, there exists $i \in \BZ$ such that $\rho$ appears in \[H_c^i(Z_{\phi, U, r}, \ov\BQ_\ell)[\phi|_{T_r^F}] = \ind^{G_r^F}_{K_{\phi, r}^F} H_c^i(Z_{\phi, U, r}^K, \ov\BQ_\ell)[\phi|_{T_r^F}].\] As $r_\phi > 0$, $(G_r^{r_\phi})^F$ belongs to $K_{\phi, r}^+$ and is normalized by $G_r^F$. It follows from Proposition \ref{natural} that $H_c^i(Z_{\phi, U, r}, \ov\BQ_\ell)[\phi|_{T_r^F}]$ is a sum of weight spaces on which $(G_r^{r_\phi})^F$ acts via the characters ${}^\g \phi^\natural$ for $\g \in G_r^F / K_{\phi, r}^F$. So the statement follows since $\rho$ is a $(G_r^{r_\phi})^F$-submodule of $H_c^i(Z_{\phi, U, r}, \ov\BQ_\ell)[\phi|_{T_r^F}]$.
\end{proof}

\subsection{Properties} Now we discuss some properties of  $\CR_{T, U, r}^G(\phi)$ parallel to those of $R_{T, U, r}^G(\phi)$.
\begin{proposition} \label{inf-CR}
    We have $\CR_{T, U, r}^G(\phi) \cong \CR_{T, U, r_\phi}^G(\phi)$ as $G_r^F$-modules.
\end{proposition}
\begin{proof}
   We can assume $r > r' \ge r_\phi$ with $r = r'+$. Let $\pi_r: Z_{\phi, U, r} \to  Z_{U, \phi, r'}$ be the natural projection. Let $x \in  Z_{U, \phi, s}$. It suffices to show \[H_c^i(\pi_r\i(x), \ov \BQ_\ell)^{(T_r^r)^F} = \begin{cases} \overline \BQ_\ell, & \text{ if } i=2 \dim G_\der; \\ 0, & \text{otherwise.} \end{cases}\] Since $r > r_\phi$, we have $(G_\der)_r^r \subseteq \CI_{\phi, U, r}$. Moreover, as $G_r^r = T_r^r (G_\der)_r^r$, it follows that \[\pi_r\i(x) \cong \{y \in G_r^r; y\i F(y) \in (G_\der)_r^r\} = (T_r^r)^F (G_\der)_r^r,\] from which the statement follows directly.
\end{proof}

Let $\pi: G_\sc \to G_\der$ be the simply connected covering. Let $T_\der = T \cap G_\der$ and $T_\sc = \pi\i(T_\der)$ be the maximal tori of $G_\der$ and $G_\sc$ respectively. Set $\phi_\der = \phi |_{T_\der^F}$ and  $\phi_\sc = \phi_\der \circ \pi$. We write $Z_{\phi_\der, U, r}^\der = Z_{G_\der, \phi_\der, U, r}$, $Z_{\phi_\sc, U, r}^\sc = Z_{G_\sc, \phi_\sc, U, r}$ and so on for simplicity.

\begin{lemma} \label{der}
    We have $Z_{\phi, U, r} = \bigsqcup_{\t \in T_r^F / (T_\der)_r^F} Z_{\phi, U, r}^\der \t$.
\end{lemma}
\begin{proof}
    It follows from the inclusion $\CI_{\phi, U, r} = \CI_{\phi_\der, U, r}^\der \subseteq (G_\der)_r$ and the natural isomorphism $T_r^F / (T_\der)_r^F \cong G_r^F / (G_\der)_r^F$.
\end{proof}

\begin{lemma} \label{bijection}
    There is a natural bijection $T_r^F / \pi((T_\sc)_r^F) \cong G_r^F / \pi((G_\sc)_r^F)$.
\end{lemma}
\begin{proof}
    It follows in the same way of \cite[Proposition 1.23]{DL}.
\end{proof}

\begin{lemma} \label{sc}
    We have $Z_{\phi, U, r} = \bigsqcup_{\t \in T_r^F / \pi((T_\sc)_r^F)} \pi(Z_{\phi_\sc, U, r}^\sc) \t$.
\end{lemma}
\begin{proof}
    Since $p \nmid |\pi_1(G_\der)|$, we have $\pi(\CI_{\phi_\sc, U, r}^\sc) = \CI_{\phi, U, r}$. Hence \[Z_{\phi, U, r} = \cup_{\g \in G_r^F / \pi((G_\sc)_r^F)} \g \pi(Z_{\phi_\sc, U, r}^\sc) = \cup_{\t \in T_r^F / \pi((T_\sc)_r^F)} \t \pi(Z_{\phi_\sc, U, r}^\sc),\] where the second equality follows from Lemma \ref{bijection}. We claim that
    \begin{align*} \tag{a} \text{ $\pi(Z_{\phi_\sc, U, r}^\sc)$ is normalized by any $\t \in T_r$ such that $\t\i F(\t) \in Z(G)_r$.} \end{align*}
    Indeed, since $T_r = (T_\der)_r Z(G)_r$, we can assume further that $\t \in (T_\der)_r$. Let $\t_\sc \in (T_\sc)_r$ such that $\pi(\t_\sc) = \t$. Then $\t_\sc\i F(\t_\sc) \in Z(G_\sc)_r$. As $\CI_{\phi_\sc, U, r}^\sc$ is normalized by $(T_\sc)_r$, it follows directly that $Z_{\phi_\sc, U, r}^\sc$ is normalized by $\t_\sc$. Hence (a) is proved.

    By (a) we deduce that
    \begin{align*} \tag{b} \text{  $Z_{\phi, U, r} = \cup_{\t \in T_r^F / \pi((T_\sc)_r^F)} \pi(Z_{\phi_\sc, U, r}^\sc) \t$.} \end{align*}
    It remains to show the union in (b) is disjoint. Suppose there exist $\t \in T_r^F$ and $g, g' \in Z_{\phi_\sc, U, r}^\sc$ such that $\pi(g) \t = \pi(g')$. We show that $\t \in \pi((T_\sc)_r^F)$. Indeed, let $\pr: G_r \to G_r / T_r$ and $\pr_\sc: (G_\sc)_r \to (G_\sc)_r / (T_\sc)_r$ be the natural projections. By identifying the quotient spaces $(G_\sc)_r / (T_\sc)_r \cong G_r / T_r$ in the natural way, we have \[\pr_\sc(g) = \pr(\pi(g)) = \pr(\pi(g) \t) = \pr(\pi(g')) = \pr_\sc(g').\] By definition, $g' = g t$ for some $t \in (T_\sc)_r$. Hence $\t = \pi(t)$. Moreover, as $g, g' \in Z_{\phi_\sc, U, r}^\sc$ it follows that $t\i F(t) \in E_{\phi_\sc, r}^{G_\sc} \cap (T_\sc)_r$. Since $E_{\phi_\sc, r}^\sc \cap (T_\sc)_r$ is connected and $F$-stable, we can write $t = t_1 t_2$ with $t_1 \in (T_\sc)_r^F$ and $t_2 \in E_{\phi_\sc, r}^\sc \cap (T_\sc)_r \subseteq (T_\sc)_r^{0+}$. Then $\pi(t_1)\i \t \in ((T_\der)_r^{0+})^F$. As $p \nmid |\pi_1(G_\der)|$, it follows from \cite[Lemma 3.13]{Kal} that $\pi$ induces a bijection $((T_\sc)_r^{0+})^F \cong ((T_\der)_r^{0+})^F$. Hence $\t \in \pi((T_\sc)_r^F)$ as desired.
\end{proof}

\begin{proposition} \label{ass-CR}
    Let $\th$ be a character of $G^F$ which is trivial over $G_\sc^F$ and $(G_\tx^{r+})^F$. Then $\CR_{T, U, r}^G(\phi) \otimes \th|_{G_r^F} \cong \CR_{T, U, r}^G(\phi \otimes \th|_{T^F})$ as $G_r^F$-modules.
\end{proposition}
\begin{proof}
    We follows the arguments of \cite[Proposition 3.7]{Ch24}. Let $g \in G_r^F$. Then we have \[\CR_{T, U, r}^G(\phi \otimes \th|_{T^F})(g) = \frac{1}{|T_r^F|} \sum_{t \in T_r^F} \tr((g, t); H_c^*(Z_{\phi, U, r}, \ov\BQ_\ell)) \phi(t)\i \th(t)\i.\] Assume $g \in t_1 \pi((G_\sc)_r^F)$ for some $t_1 \in T_r^F$. Let $t \in T_r^F$. If $t_1 t \notin \pi((T_\sc)_r^F)$, then the action of $(g, t)$ on \[Z_{\phi, U, r} = \bigsqcup_{\t \in T_r^F / \pi((T_\sc)_r^F)} \pi(Z_{\phi_\sc, U, r}^\sc) \t\] sends each component $\pi(Z_{\phi_\sc, U, r}^\sc) \t$ to a different one. Hence \[\tr((g, t); H_c^*(Z_{\phi, U, r}, \ov\BQ_\ell)) = 0.\] If $t_1 t \in \pi((T_\sc)_r^F)$, we have $\th(g) = \th(t_1) = \th(t)\i$ as $\th \circ \pi$ is trivial over $G_\sc^F$. Therefore, it is always true that \begin{align*}\CR_{T, U, r}^G(\phi \otimes \th|_{T^F})(g) &= \frac{1}{|T_r^F|} \sum_{t \in T_r^F} \tr((g, t); H_c^*(Z_{\phi, U, r}, \ov\BQ_\ell)) \phi(t)\i \th(g) \\ &= \CR_{T, U, r}^G(\phi)(g) \th(g) \\ &= (\CR_{T, U, r}^G(\phi) \otimes \th|_{G_r^F})(g). \end{align*} Thus $\CR_{T, U, r}^G(\phi) \otimes \th|_{G_r^F} \cong \CR_{T, U, r}^G(\phi \otimes \th|_{T^F})$ and the proof is finished.
\end{proof}

\subsection{Restrictions to Levi subgroups} \label{subsec:Levi}
Let $T$ and $\phi$ be as in \S \ref{subsec:Howe}. Let $L \supseteq T$ be a Levi subgroup of $G$ over $k$. Let $(G^i, \phi_i, r_i)_{-1 \le i \le d_\phi}$ be a Howe factorization of $\phi$ in $G$. We show that it induces a Howe factorization of $\phi$ in $L$ in a natural way.

First note that there is a unique sequence of integers \[0 = j_0 < j_1 < \cdots < j_d < j_{d+1} = d_\phi + 1\] with $0 \le d \le d_\phi$ such that \begin{itemize}
    \item $G^{j_i - 1} \cap L \subsetneq G^{j_i} \cap L$ for $1 \le i \le d$;

    \item $G^{j_i} \cap L = G^{j_i + 1} \cap L = \cdots =  G^{j_{i+1} - 1} \cap L$ for $0 \le i \le d$.
\end{itemize}
We define $d_\phi^L = d$, $L^i = G^{j_{i+1}-1} \cap L$, $\phi_i^L = \prod_{k = j_i}^{j_{i+1}-1} \phi_k|_{L_i^F}$ and $r_i^L = r_{j_{i+1} - 1}$ for $0 \le i \le d_\phi^L$. Moreover, we set $L^{-1} = T$, $\phi_{-1}^L = \phi_{-1}$ and $r_{-1}^L = 0$.
\begin{lemma} \label{Howe-Levi}
    Let notation be as above. Then $(L^i, \phi_i^L, r_i^L)_{-1 \le i \le d_\phi^L}$ is a Howe factorization of $\phi$ in $L$. Moreover, $\CI_{G, \phi, U} \cap L = \CI_{L, \phi, U \cap L} (E_{G, \phi} \cap T)$ and $E_{G, \phi} \cap L = E_{L, \phi} (E_{G, \phi} \cap T)$.
\end{lemma}
\begin{proof}
    The first statement follows by definition. The second one follows from the equalities $K_{G, \phi} \cap U \cap L = K_{L, \phi} \cap U$ and $K_{G, \phi}^+ \cap \ov U \cap L = K_{L, \phi}^+ \cap \ov U$ together with the inclusions $E_{L,\phi} \cap T  \subseteq E_{G, \phi} \cap T \subseteq \CI_{G, \phi, U} \cap L$.
\end{proof}

We consider the following intersection \[Z_{\phi, U, r}^L := Z_{G, \phi, U, r} \cap L_r = \{g \in L_r; g\i F(g) \in F(\CI_{G, \phi, U, r} \cap L_r)\},\] which also admits a natural action of $L_r^F \times T_r^F$ by left/right multiplication. The isotropic subspace \[\CR_{T, U \cap L, r}^{L \subseteq G}(\phi) := H_c^*(Z_{\phi, U, r}^L, \ov \BQ_\ell)[\phi|_{T_r^F}]\] is a virtual representation of $L_r^F$.
\begin{proposition} \label{Levi-CR}
    We have $\CR_{T, U \cap L, r}^L(\phi) \cong \CR_{T, U \cap L, r}^{L \subseteq G}(\phi)$ as virtual representations of $L_r^F$.
\end{proposition}
\begin{proof}
    By Proposition \ref{natural} we have \begin{align*} \CR_{T, U \cap L, r}^L(\phi) &= H_c^*(Z_{L, \phi, U \cap L, r} / E_{L, \phi, r}, \ov\BQ_\ell)[\phi|_{T_r^F}]; \\ \CR_{T, U \cap L, r}^{L \subseteq G}(\phi) &= H_c^*(Z_{\phi, U, r}^L / (E_{G, \phi, r} \cap L_r), \ov\BQ_\ell)[\phi|_{T_r^F}].\end{align*} By Lemma \ref{Howe-Levi}, the natural morphism \[\pr: Z_{L, U \cap L, r} / E_{L, \phi, r} \to Z_{\phi, U, r}^L / (E_{G, \phi, r} \cap L_r)\] is a $(E_{G, \phi, r} \cap T_r)^F / (E_{L, \phi, r} \cap T_r)^F$-torsor. Thus for any $i \in \BZ$ we have \[H_c^i(Z_{\phi, U, r}^L / (E_{G, \phi, r} \cap L_r), \ov\BQ_\ell) \cong H_c^i(Z_{L, \phi, U \cap L, r} / E_{L, \phi, r}, \ov\BQ_\ell)^{(E_{G, \phi, r} \cap T_r)^F}.\] Since $\phi$ is trivial over $(E_{G, \phi, r} \cap T_r)^F$, it follows that \[H_c^i(Z_{\phi, U, r}^L / (E_{G, \phi, r} \cap L_r), \ov\BQ_\ell)[\phi|_{T_r^F}] \cong H_c^i(Z_{L, \phi, U \cap L, r} / E_{L, \phi, r}, \ov\BQ_\ell)[\phi|_{T_r^F}].\] Hence the statement follows.
\end{proof}

\section{A inner product formula} \label{sec-product}
We keep notations in \S \ref{sec-CR}. Let $S$ be another $k$-rational maximal torus which is conjugate to $T$ by $G_\tx$. We fix two smooth characters $\phi$ and $\psi$ of $T^F$ and $S^F$ respectively. Let $(G^i, \phi_i, r_i)_{-1 \le i \le d_\phi}$ and $(M^i, \psi_i, s_i)_{-1 \le i \le d_\psi}$ be Howe factorizations of $\phi$ and $\psi$ respectively. Fix $r \ge \max\{r_\phi, r_\psi\}$.

Let $(U, \ov U)$ and $(V, \ov V)$ be two pairs of opposite maximal unipotent subgroups of $G$, which are normalized by $T$ and $S$ respectively. Moreover, we may assume that $(M^i)_{0 \le i \le d_\psi}$ is standard with respect to $V$, that is, for each $-1 \le i \le d_\psi$, $M^i$ and $V$ generates a parabolic subgroup $P^i = M^i N^i$ of $G$, where $M^i$ is the Levi part and $N^i \subseteq V$ is the unipotent radical. Note that such $V$ always exists by Lemma \ref{standard-U}.

Let $N_{G_r}(T_r, S_r) = \{x \in G_r; {}^x T_r = S_r\}$ and $W_{G_r}(T_r, S_r) = N_{G_r}(T_r, S_r) / S_r$. The goal of this section is to prove the following inner product formula.
\begin{theorem} \label{CR-R}
Assume that $S$ is elliptic. Then \[\<\CR_{T, U, r}^G(\phi), R_{S, V,  r}^G(\psi)\>_{G_r^F} = \sharp \{{x \in W_{G_r}(T_r, S_r)^F}; {}^x \phi|_{T_r^F} = \psi|_{S_r^F}\}.\]
\end{theorem}

To compute the left hand side of Theorem \ref{CR-R}, we consider the variety \[\Sigma = \{(x, x', y) \in F\CI_{\phi, U, r}\times FV_r \times G_r; x F(y) = y x'\}.\] It admits a $T_r^F \times S_r^F$-action given by $(t, s): (x, x', y) \mapsto (txt\i, s x' s\i, t y s\i)$. We write \[H_c^*(\Sigma, \ov \BQ_\ell)_{\phi, \psi\i} = H_c^*(\Sigma, \ov \BQ_\ell)[\phi|_{T_r^F} \boxtimes \psi|_{S_r^F}\i]\] for the alternating sum of the corresponding isotropic subspaces. Following \cite[\S 6.6]{DL}, we have \[\<\CR_{T, U, r}^G(\phi), R_{S, V,  r}^G(\psi)\>_{G_r^F} = \dim H_c^*(\Sigma, \ov \BQ_\ell)_{\phi, \psi\i}.\] Hence it remains to compute $H_c^*(\Sigma, \ov \BQ_\ell)_{\phi, \psi\i}$.

\subsection{A decomposition of $\Sigma$} \label{subsec:sigma}
Set $M = M^{d_\psi -1}$ and $N = N^{d_\psi - 1}$. Let $\ov N$ be the opposite of $N$. Then \[G_r = \bigsqcup_{w \in M_r \backslash M_r N_{G_r}(T_r, S_r)} G_{w, r},\] where $G_{w, r} = \CZ_{w, r} \dw\i M_r N_r$, $\CZ_{w, r} = U_r \ov U_r^{0+} \cap {}^{\dw\i} \ov N_r$ and $\dw \in M_r N_{G_r}(T_r, S_r)$ is a lift of $w$. This induces a decomposition \[\Sigma = \bigsqcup_{w \in N_{G_r}(T_r, S_r) M_r / M_r} \Sigma_w,\] where $\Sigma_w = \{(x, x', y) \in \Sigma; y \in G_{w, r}\}$. As each $\Sigma_w$ is $T_r^F \times T_r^F$-stable, we have \[H_c^*(\Sigma, \ov \BQ_\ell)_{\phi, \psi\i} = \sum_w H_c^*(\Sigma_w, \ov \BQ_\ell)_{\phi, \psi\i}.\] To study $\Sigma_w$, let \[\hat \Sigma_w = \{(x, x', z, m, n) \in F\CI_{\phi, U, r} \times FV_r \times \CZ_{w, r} \times M_r \times N_r; x F(z \dw\i m n) = z \dw\i m n x'\}.\] It has a $T_r^F \times S_r^F$-action given by \[(t, s): (x, x', z, m, n) \mapsto (txt\i, s x' s\i, t z t\i, w(t) m s\i, sns\i).\] Then the map $\hat \Sigma_w \to \Sigma$ given by $(x, x', z, m, n) \mapsto (x, x', z \dw\i m n)$ is a $T_r^F \times S_r^F$-equivariant affine space fibration. In particular, \[H_c^*(\Sigma_w, \ov \BQ_\ell)_{\phi, \psi\i} \cong H_c^*(\hat \Sigma_w, \ov \BQ_\ell)_{\phi, \psi\i}.\] Since $N \subseteq V$, using the substitution $x'F(n) \mapsto x'$ we can rewrite $\hat\Sigma_w$ as \[\hat \Sigma_w = \{(x, x', z, m, n) \in F\CI_{\phi, U, r} \times FV_r \times \CZ_{w, r} \times M_r \times N_r; x F(z \dw\i m) = z \dw\i m n x'\},\] on which the action of $T_r^F \times S_r^F$ is unchanged.

Write $\hat\Sigma_w = \hat\Sigma_w' \sqcup \hat\Sigma_w''$, where $\hat \Sigma_w', \hat\Sigma''$ are $T_r^F \times S_r^F$-stable locally closed subsets defined by
\begin{align*}
    \hat\Sigma_w' &= \{(x, x', z, m, n) \in \hat\Sigma_w; z \in \CZ_{w,r} \setminus \CI_{\phi, U, r}\}; \\  \hat\Sigma_w'' &= \{(x, x', z, m, n) \in \hat\Sigma_w; z \in \CZ_{w, r} \cap \CI_{\phi, U, r}\}.
\end{align*}
In particular, $H_c^*(\hat \Sigma_w, \ov \BQ_\ell)_{\phi, \psi\i} = H_c^*(\hat \Sigma_w'', \ov \BQ_\ell)_{\phi, \psi\i} + H_c^*(\hat \Sigma_w', \ov \BQ_\ell)_{\phi, \psi\i}$.

\subsection{The first computation}
First we compute $H_c^*(\hat \Sigma_w'', \ov \BQ_\ell)_{\phi, \psi\i}$. The result is as follows.
\begin{lemma} \label{red-Levi}
    We have $H_c^*(\hat\Sigma_w'', \ov \BQ_\ell)_{\phi, \psi\i} \neq 0$ only if $w = F(w)$. In this case, \[\dim H_c^*(\hat\Sigma_w'', \ov\BQ_\ell)_{\phi, \psi\i} = \<\CR_{{}^\dw T, {}^\dw U \cap M, r}^M({}^\dw\phi), R_{S, V \cap M, r}^M(\psi)\>_{M_r^F},\] where $\dw \in G_r^F \cap M_r N_{G_r}(T_r, S_r)$ is a lift of $w$.
\end{lemma}
\begin{proof}
    By the substitution $x F(z) \mapsto x$ we can write $\hat\Sigma_w''$ as \[\hat\Sigma_w'' = \{(x, x', z, m, n) \in F\CI_{\phi, U, r} \times FV_r \times (\CZ_{w, r} \cap \CI_{\phi, U, r}) \times M_r \times N_r; x F(\dw\i m) = z \dw\i m n x'\}.\] Consider the algebraic group \[D_w = \{(t, s) \in T_r \times S_r; F(\dw) t\i F(t) F(\dw)\i = s\i F(s) \in Z(M)_r^\circ\},\] where $T_{r, \red}, S_{r, \red}$ are the reductive subgroups of $T_r, S_r$ respectively. The action of $D_w$ on $\hat\Sigma_w''$ is given by \[(t, s): (x, x', z, m, n) \mapsto (txt\i, sx's\i, tzt\i, \dw t \dw\i m s\i, s n s\i).\] Since the actions of $D_w$ and $T_r^F \times S_r^F$ on $\hat\Sigma_w''$ commute with each other, we have \[H_c^*(\hat\Sigma_w'', \ov \BQ_\ell)_{\phi, \psi\i} \cong H_c^*((\hat\Sigma_w'')^{D_w^\circ}, \ov \BQ_\ell)_{\phi, \psi\i}.\] As $F$ preserves $Z(M)^\circ$, the image of the natural projection $D_w^\circ \to S_{r, \red}$ is $Z(M)_{r, \red}^\circ$.

    Assume $(\hat\Sigma_w'')^{D_w^\circ} \neq \emptyset$, and let $(x, x', z, m, n) \in (\hat\Sigma_w'')^{D_w^\circ}$. Then we have $m = \dw t \dw\i m s\i = \dw t \dw\i s\i m$ for $(t, s) \in D_w^\circ \subseteq T_r \times Z(M)_r^\circ$. This implies that $t = \dw\i s \dw$ and $D_w^\circ = \{(w\i(s), s); s \in Z(M)_r^\circ\}$. Thus $\dw F(\dw)\i \in M_r$ and \[(\hat\Sigma_w'')^{D_w^\circ} \subseteq {}^{\dw\i} M_r \times M_r \times \{1\} \times M_r \times \{1\}.\] So we may assume $\dw = F(\dw)$ and deduce that \begin{align*}(\hat\Sigma_w'')^{D_w^\circ} &= \{(x, x', m) \in (F\CI_{\phi, U, r} \cap {}^{\dw\i} M_r) \times (FV_r \cap M_r) \times M_r; x \dw\i F(m) = \dw\i m x'\} \\ &\cong \{(x, x', m) \in (F({}^\dw\CI_{\phi, U, r} \cap M_r) \times F(V_r \cap M_r) \times M_r; x F(m) = m x'\}.\end{align*}
    Noticing that ${}^\dw\CI_{\phi, U, r} = \CI_{{}^\dw U, {}^\dw\phi, r}$, we have \begin{align*} \dim H_c^*((\hat\Sigma_w'')^{D_w^\circ}, \ov\BQ_\ell)_{\phi, \psi\i} &= \<\CR_{{}^\dw T, {}^\dw U, r}^{M \subseteq G}({}^\dw \phi), R_{T, V \cap M, r}^M(\psi)\>_{M_r^F} \\ &=\<\CR_{{}^\dw T, {}^\dw U \cap M, r}^M({}^\dw \phi), R_{T, V \cap M, r}^M(\psi)\>_{M_r^F},\end{align*} where the first equality follows from the definition of $\CR_{T^\dw, U^\dw, r}^{M \subseteq G}(\phi^\dw)$ in \S \ref{subsec:Levi}, and the second one follows from Proposition \ref{Levi-CR}. The proof is finished.
\end{proof}

\subsection{A vanishing result} \label{subsec: tech} Now we compute $H_c^*(\hat \Sigma_w'', \ov \BQ_\ell)_{\phi, \psi\i}$ following the strategies of \cite{L04} and \cite{CS17}.

Let $\Phi = \Phi(G, T)$ and $\Psi = \Phi(G, S)$. Let $\Psi_{\ov N} \subseteq \Psi$ be the set of roots appearing in $\ov N$. Let $\Phi^+$ be the set of (positive) roots appearing in $U$. Let $\preceq$ be the dominance order on $\Phi$ induced from $\Phi^+$, namely, $\a \preceq \b$ if and only if $\b - \a$ is a sum of roots in $\Phi^+$. Fix $\dw \in N_r(T_r, S_r)$ and set $\D_w = {}^{\dw\i} \Psi_{\ov N} \subseteq \Phi$.

Recall that $(G^i, \phi_i, r_i)_{-1 \le i \le d_\phi}$ is a Howe factorization of $\phi$. Let $\Phi_i = \Phi(G^i, T)$ for $0 \le i \le d_\phi$. Let $\g \in \Phi$. Let $i(\g) = i^\phi(\g)$ and $r(\g) = r^\phi(\g)$ be defined as in \S\ref{subsec:Howe}. Then we have \[\CZ_{w, r} = \prod_{\g \in \D_w} (G^\g)_r^{\varepsilon(\g)} \text{ and } \CZ_{w,r} \cap \CI_{\phi, U, r} = \prod_{\g \in \D_w} (G^\g)_r^{\e(\g)},\] where $\varepsilon(\g) = 0$, $\e(\g) = \e^\phi(\g) = r(\g)/2$ if $\g \in \Phi^+$ and $\varepsilon(\g) = 0+$, $\e(\g) = \e^\phi(\g) = r(\g)/2+$ otherwise. There is a decomposition \[ \CZ_{w, r} \setminus \CI_{\phi, U, r} = \bigsqcup_{\emptyset \neq I \subseteq \D_w,~ v \in \BR_{\ge 0}^{\D_w}}  \ \CZ_{w, r}^{I, v},\] where $\CZ_{w, r}^{I, v}$ consists of elements $z = \prod_{\g \in \D_w} z_\g$ such that $z_\g \in (G^\g)_r^{v(\g), *}$ with $\varepsilon(\g) \le v(\g) < \e(\g)$ for $\g \in I$ and $x_\g \in (G^\g)_r^{v(\g)}$ with $v(\g) = \e(\g)$ otherwise. Here $(G^\g)_r^{r', *} = (G^\g)_r^{r'} \setminus (G^\g)_r^{r'+}$ for $0 \le r' \le r$.

Let $\emptyset \neq I \subseteq \D_w$ and $v \in \BR_{\ge 0}^I$. Set $\d(I, v) = \max \{r(\g) - v(\g); \g \in I\}$ and \[c(I, v) = \{\g \in I; r(\g) - v(\g) = \d(I, v)\}.\]
\begin{lemma} \label{commutator}
Let $I, v$ be as above. Let $\a \in \min_{\preceq} c(I, v)$. Let $y \in \CZ_{w, r}^{I, v}$ and let $\z \in (G^{-\a})_r^{r(\a) - v(\a)}$. Then there exist $\o_{y, \z} \in \CI_{\phi, U, r}$, and $\t_{\z, y} \in (T^\a)_r^{r(\a)}$ such that \[y \z = \o_{\z, y} \t_{\z, y} y.\] Moreover, the map $\z \mapsto \t_{\z, y}$ induces an isomorphism of algebraic groups $\l_y: (G^{-\a})_r^{r(\a) - v(\a)} / (G^{-\a})_r^{(r(\a) - v(\a))+} \overset \sim \longrightarrow (T^\a)_r^{r(\a)} / (T^\a)_r^{r(\a)+}$.
\end{lemma}
\begin{proof}
First we note that
\[\tag{a} \text{$(G^\g)_r^s \subseteq \CI_{\phi, U, r}$ $\iff$ $\g \in \Phi^+$, $s \ge r(\g)/2$ or $-\g \in \Phi^+$, $s > r(\g)/2$.}\]
Since $\varepsilon(\g) \le v(\g)$ for $\g \in \D_w$, we have
\[\tag{a'} \text{ $v(\g) \ge 0$ if $\g \in \Phi^+$ and $v(\g) > 0$ otherwise.}\]
Moreover, as $\varepsilon(\g) \le v(\g) < \e(\g)$ for $\g \in I$, if follows that
\[\tag{b} \text{ $r(\g) > 0$, $v(\g) \le r(\g)/2$ and $(G^{-\g})_r^{r(\g) - v(\g)} \subseteq \CI_{\phi, U, r}$ for $\g \in I$.} \]

Note that $y \z y\i$ is a product of elements $x = x_{\b, h, n, \overset \rightarrow \g}$, where $\b \in \Phi$, $h \in \BR_{>0}$, $n \in \BZ_{\ge 1}$ and $ \overset \rightarrow \g = (\g_i)_{1 \le i \le m}$ is a sequence of roots in $\D_w$ such that
\begin{itemize}
    \item $i(\g_1) \ge i(\g_2) \ge \cdots \ge i(\g_m)$;

    \item $h = n(r(\a) - v(\a)) + \sum_{i=1}^m v(\g_i)$;

    \item either $x \in (G^\b)_r^h$ with $\b = -n\a + \sum_{i=1}^m \g_i \in \Phi$ or $-n\a + \sum_{i=1}^m \g_i = 0$ and $x \in (T^\b)_r^h$ with $\b = \g_{i_0}$ for some $1 \le i_0 \le m$.
\end{itemize}

To show the first statement, it suffices to show $x \in \CI_{\phi, U, r}$ unless $n = 1$, $m = 1$ and $\g_1 = \a$. Suppose that $x \notin \CI_{\phi, U, r}$. We show it will leads to a contradiction.

Note that
\[\tag{c} \text{ $v(\g) > 0$ if $i(\a) < i(\g)$.} \]
Indeed, assume $i(\a) < i(\g)$, then $r(\g) > r(\a) > 0$ by (b). If $\g \in I$, as $\a \in \d(I, v)$ we have $r(\g) - v(\g) \le r(\a) - v(\a)$ and hence $v(\g) > v(\a) \ge 0$. Otherwise, by (a) we have $v(\g) \ge r(\g)/2 > r(\a)/2 > 0$. Hence (c) always holds.

Now we claim that
\[\tag{d} \text{ $i(\g_1) \le i(\a)$ and hence $\b \in \Phi_{i(\a)}$.} \]
Indeed, assume that $i(\g_1) > i(\a)$. Then $r(\g_1) > 0$ (by (c)) and $\b \in \Phi_{i(\g_1)}$. If $\g_1 \in I$, we have \[h \ge r(\a) - v(\a) + v(\g_1) \ge r(\g_1) - v(\g_1) + v(\g_1) = r(\g_1).\] Otherwise, by (a) and (b) we have \[h \ge r(\a) - v(\a) + r(\g_1)/2 > r(\a)/2 + r(\g_1)/2.\] In a word, $h \ge \e(\g_1)$. Since $x \notin \CI_{\phi, U, r}$, we have $-n\a + \sum_{i=1}^e \g_i = 0$ and $x \in (T^\b)_r^h$. In particular, $m \ge 2$ and $i(\g_2) = i(\g_1) > i(\a)$. Hence by previous computation we always have \[h \ge (r(\a) - v(\g)) + v(\g_1) + v(\g_2) >  r(\g_1),\] This means $x \in \CI_{\phi, U, r}$, a contradiction. So (d) is proved.

Then we claim that $n=1$. Indeed, assume $n \ge 2$. By (b) we have $h \ge 2(r(\a) - v(\a)) \ge r(\a)$. As $x \notin \CI_{\phi, U, r}$, we have $h = r(\a)$, $-n\a + \sum_{i=1}^m \g_i = 0$ and $x \in (T^\b)_r^h$. By (a), (a') and (b) this means that $n = 2$, $r(\a) - v(\a)  = r(\a)/2$ and $v(\g_i) = 0$ for $1 \le i \le m$. Hence $-\a, \g_i \in \Phi^+$, contradicting that $-n\a + \sum_{i=1}^m \g_i = 0$.

Third we claim that $i(\g_1) = i(\a)$. Assume otherwise, by (d) we have $i(\g_i) < i(\a)$ for $1 \le i \le m$. Moreover, as $n=1$ we have $\b = -\a + \sum_{i=1}^m \g_i \in \Phi_{i(\a)}$ and $x \in (G^\b)_r^h$. By (b) and (a) and (a') we have either $h > r(\a)/2$ or $h = r(\a)/2$ and $\b \in \Phi^+$. In either case we have $x \in \CI_{\phi, U, r}$, a contradiction.

Since $i(\g_1) = i(\a)$, by the proof of (d) we always have \[h \ge (r(\a) - v(\a)) + v(\g_1) \ge r(\g_1) = r(\a).\] Again, since $\b \in \Phi_{i(\a)}$ and $x \notin \CI_{\phi, U, r}$, we have $h = r(\a)$ and $-n\a + \sum_{i=1}^m \g_i = 0$. In particular, $v(\g_i) = 0$ and hence $\g_i \in \Phi^+$ for $2 \le i \le m$. Assume $\g_1 \notin I$. Then $v(\g_1) \ge r(\g_1)/2 = r(\a)/2$, which implies that $v(\g_1) = v(\a) = r(\a)/2$ (since $r(\a) - v(\a) \ge r(\a)/2$) and hence $-\a, \g_1 \in \Phi^+$. This contradicts that $-n\a + \sum_{i=1}^m \g_i = 0$. So we have $\g_1 \in c(I, v)$. Note that \[\a - \g_1 = \sum_{i=2}^m \g_i \in \BZ_{\ge 0} \Phi^+.\] As $\a \in \min_{\preceq} c(I, v)$, it follows that $\g_1 = \a$ and $m = 1$. So the first statement is proved. The second one follows by a direct computation.
\end{proof}

\begin{lemma} \label{vanish}
    We have $H_c^*(\hat\Sigma_w', \overline \BQ_\ell)_{\phi, \psi\i} = 0$.
\end{lemma}
\begin{proof}
We have a decomposition \[\hat\Sigma_w' = \bigsqcup_{I \subseteq \D_w,~ v \in \BR_{\ge 0}^{\D_w}} \hat\Sigma_w^{I, v},\] where $\hat\Sigma_w^{I, v}$ consists of $(x, x', y, m, n) \in \hat\Sigma_w'$ with $y \in \CZ_{w, r}^{I, v}$ such that \[ x F(y \dw\i m) = y \dw\i m n x'.\] Then action of $T_r^F \cong T_r^F \times \{1\} \subseteq T_r^F \times S_r^F$ on $\hat\Sigma_w^{I, v}$ is given by \[t: (x, x', y, m, n) \mapsto (txt\i, x', t y t\i, \dw t \dw\i m, n).\]

It suffices to show the $\phi$-isotropic subspace $H_c^*(\hat\Sigma_w^{I, v}, \ov \BQ_\ell)_\phi$ for $T_r^F$ is trivial. Let $\a \in \min_{\preceq} c(I, v)$. By Lemma \ref{commutator}, for $y \in \CZ_{w, r}^{I, v}$ and $\z \in (G^{-\a})_r^{r(\a) - v(\a)}$, there exist $\t_{\z, y} \in (T^\a)_r^{r(\a)}$ and $\o_{\z, y} \in \CI_{\phi, U, r}$ such that \[y \z = \o_{\z, y} \t_{\z, y} y.\] Consider the natural quotient maps \begin{gather*}\th_1: (G^{-\a})^{r(\a) - v(\a)} \to (G^{-\a})_r^{r(\a) - v(\a)} / (G^{-\a})_r^{(r(\a) - v(\a))+}; \\ \th_2: (T^\a)_r^{r(\a)} \to (T^\a)_r^{r(\a)} / (T^\a)_r^{r(\a)+}. \end{gather*} Let $\vartheta_1$ be a section of $\th_1$ such that $\th_1 \circ \vartheta_1 = \id$ and $\vartheta_1(1) = 1$.

Consider the following algebraic group \[\CD = \{t \in T_r^{r(\a)}; t\i F\i(t) \in (T^\a)_r^{r(\a)}\}.\] For $t \in \CD$ we define an isomorphism $f_t: \hat\Sigma_w^{I, v} \to \hat\Sigma_w^{I, v}$ by \[f_t(x, x', y, m, n) = (x_t, x_t', y_t, m_t, n_t) := (x_t, x' F({}^{m\i \dw} \z), t y t\i, \dw t \dw\i m, n),\] where $\z = \vartheta_1\l_y\i \th_2 (t F\i(t)\i) \in (G^{-\a})_r^{r(\a) - v(\a)}$, $\l_y$ is as in Lemma \ref{commutator}, and $x_t \in F\CI_{\phi, U, r}$ is determined by the equality \[x_t F(y_t \dw\i m_t) = y_t \dw\i m_t n_t x_t'.\] By Lemma \ref{commutator} and that ${}^\dw \z \in N_r$ one checks that $f_t$ is well defined.

By general principle, the induced map of $f_t$ on each $H_c^i(\hat\Sigma_w^{I, v}, \ov\BQ_\ell)$ is trivial for $t \in N_F^{F^n}(((T^\a)_r^{r(\a)})^{F^n}) \subseteq \CD^\circ$, where $n \in \BZ_{\ge 1}$ such that $F^n(T^\a) = T^\a$. On the other hand, as $\a \in I$ we have $i(\a) \ge 1$ and \[\phi|_{N_F^{F^n}(((T^\a)_r^{r(\a)})^{F^N})} = \phi_{i(\a)-1}|_{N_F^{F^n}(((T^\a)_r^{r(\a)})^{F^n}) },\] which is nontrivial since $\phi_{i(\a)-1}$ is $(G_{i(\a)-1}, G_{i(\a)})$-generic. Thus it follows that $H_c^*(\hat\Sigma_w^{I, v}, \ov\BQ_\ell)_\phi = 0$ as desired.
\end{proof}

\subsection{End of the proof} Now we are ready to show the main result of this section.
\begin{proof} [Proof of Theorem \ref{CR-R}]
By \cite[Proposition 3.7]{Ch24} and Proposition \ref{ass-CR} we have \[\<\CR_{T, U, r}^G(\phi), R_{S, V,  r}^G(\psi)\>_{G_r^F} = \<\CR_{T, U, r}^G(\phi \otimes \psi_{d_\psi}\i|_{T^F}), R_{S, V,  r}^G(\psi \otimes \psi_{d_\psi}\i|_{S^F})\>_{G_r^F}.\] Thus by replacing $\phi$ and $\psi$ with $\phi \otimes \psi_{d_\psi}\i|_{T^F}$ and $\psi \otimes \psi_{d_\psi}\i|_{S^F}$ respectively, we can assume further that $\psi$ has depth $s_{d_\psi-1}$.

We argue by induction on $d_\psi$ and the semisimple rank of $G$. If $G = T$, then $X_{U, r} = Z_{\phi, U, r} = T_r^F$ and the statement is trivial. Suppose that $d_\psi = 0$, then $\psi$ has depth $r_\psi = s_{-1} = 0$. Moreover, as $S$ is elliptic, it follows from Theorem \ref{degenracy} that \[R_{S, V, r}^G(\psi) \cong R_{S, V, 0}^G(\psi).\] In particular, $R_{S, V, r}^G(\psi)$ is a linear combination of irreducible $G_r^F$-modules on which $(G_r^{0+})^F$ acts trivially. Now we first assume that $r_\phi > 0$. By Corollary \ref{restriction}, $\CR_{T, U, r}^G(\phi)$ is a linear combination of irreducible $G_r^F$-modules on which $(G_r^{r_\phi})^F$ acts via nontrivial characters $({}^\g \phi^\natural) |_{(G_r^{r_\phi})^F}$ for $\g \in G_r^F$. Thus \[\<\CR_{T, U, r}^G(\phi), R_{S, V, r}^G(\psi)\>_{G_r^F} = 0 = \sharp \{{x \in W_{G_r}(T_r, S_r)^F}; {}^x\phi = \psi\}\] as desired. Now assume that $r_\phi = 0$. Then $Z_{U, \phi, 0} = X_{U, 0}$ and it follows from Proposition \ref{inf-CR} that \[\CR_{T, U, r}^G(\phi) = \CR_{T, U, 0}^G(\phi) = R_{T, U, 0}^G(\phi).\] By \cite[Theorem 6.8]{DL} we have \begin{align*}\<\CR_{T, U, r}^G(\phi), R_{S, V, r}^G(\psi)\>_{G_r^F} &= \sharp \{{x \in W_{G_0}(T_0, S_0)^F}; {}^x \phi|_{T_0^F} = \psi |_{S_0^F}\} \\ &= \sharp \{{x \in W_{G_r}(T_r, S_r)^F}; {}^x \phi|_{T_r^F} = \psi|_{S_r^F}\}.\end{align*} So the statement is true when $r_\psi = 0$.

Suppose that $d_\psi \ge 1$. Let $M = M^{d_\psi - 1}$. Then \begin{align*} &\quad\ \<\CR_{T, U, r}^G(\phi), R_{S, V, r}^G(\psi)\>_{G_r^F} \\ &= \sum_{w \in M_r \backslash M_r N_{G_r}(T_r, S_r)} \dim H_c^*(\hat \Sigma_w, \ov\BQ_\ell)_{\phi, \psi\i}  \\ &=\sum_{w \in (M_r \backslash M_r N_{G_r}(T_r, S_r))^F} \<\CR_{{}^\dw T, {}^\dw U \cap M, r}^M({}^\dw \phi), R_{S, V \cap M, r}^M(\psi)\>_{M_r^F} \\ &= \sum_{w \in (M_r \backslash M_r N_{G_r}(T_r, S_r))^F} \sharp\{u \in W_{M_r}(T_r^\dw, S_r)^F; {}^{\dot u \dw} \phi |_{T_r^F} = \psi |_{S_r^F}\} \\ &= \sharp \{{x \in W_{G_r}(T_r, S_r)^F}; {}^x \phi |_{T_r^F} = \psi |_{S_r^F}\}, \end{align*} where the second equality follows from Lemma \ref{red-Levi} and Lemma \ref{vanish}, and the third one follows by induction hypothesis for $M$. The proof is finished.
\end{proof}

\section{Coincidence of the two representations} \label{sec-algebraic}
Let $\tx$, $T$, $\phi$, $r_\phi$, $(G^i, \phi_i, r_i)_{-1 \le i \le d_\phi}$ and $(U, \ov U)$ be as in \S\ref{sec-product}. Assume that $(G^i)_{0 \le i \le d_\phi}$ is standard with respect to $U$. Let $r \ge r_\phi$ and let $K_{\phi, r}$, $H_{\phi, r}$, $K_{\phi, r}^+$ and $E_{\phi, r}$ be as in \S \ref{sec-CR}. Let $\Phi = \Phi(G, T)$ and denote by $\Phi^+ = - \Phi^-$ the set of roots in $\Phi$ appearing in $U$. Set $L = G^0$.

In this section, we compute the self inner product $\<\CR_{T, U, r}^G(\phi), \CR_{T, U, r}^G(\phi)\>_{G_r^F}$ and show that $\CR_{T, U, r}^G(\phi) = \CR_{T, U, r}^G(\phi)$.

\subsection{The representation $H_c^*(Z_{\phi, U, r}^K, \ov\BQ_\ell)[\phi]$} Let $Z_{\phi, U, r}^K = Z_{\phi, U, r} \cap K_{\phi, r}$ be defined in \S \ref{subsec:CR}. We show that the self inner product of $\CR_{T, U, r}^G(\phi)$ equals to the self inner product of $K_{\phi, r}^F$-representation \[H_c^*(Z_{\phi, U, r}^K, \ov\BQ_\ell)[\phi] := H_c^*(Z_{\phi, U, r}^K, \ov\BQ_\ell)[\phi|_{T_r^F}].\]

For $0 \le i \le d_\phi$ let $\hat \phi_i$ be the character of $(G^i)_\tx^F (G_\tx^{r_i/2+})^F$ defined in \cite[\S 4]{Yu}.
\begin{lemma} \label{intertwine-i}
    Let $0 \le i \le d_\phi-1$ and $g \in (G^{i+1})_\tx^F$. If $g$ intertwines $\hat \phi_i |_{((G^{i+1})_\tx^{r_i, r_i/2+})^F}$, then $g \in ((G^{i+1})_\tx^{r_i/2})^F (G^i_\tx)^F ((G^{i+1})_\tx^{r_i/2})^F$.
\end{lemma}
\begin{proof}
     Note that $(G^{i+1})_\tx \cap G^i = (G^i)_\tx$ since $T \subseteq G^i$ is an unramified maximal torus. Then the statement follows from \cite[Theorem 9.4]{Yu}.
\end{proof}

\begin{lemma} \label{intertwine}
    Let $\rho, \rho'$ be two irreducible summands of $H_c^*(Z_{\phi, U, r}^K, \ov \BQ_\ell)[\phi]$ as $K_{\phi, r}^F$-modules. Let $g \in G_r^F$ be such that \[\hom_{K_{\phi, r}^F \cap {}^g K_{\phi, r}^F}({}^g \rho, \rho') \neq \{0\}.\] Then $g \in K_{\phi, r}^F$ and hence $\rho \cong \rho'$. Recall that ${}^g \rho(x) = \rho(g\i x g)$.
\end{lemma}
\begin{proof}
    By Lemma \ref{natural}, $(K_{\phi, r}^+)^F$ acts on $\rho$ and $\rho'$ by the character $\phi^\natural$. In particular, $g$ intertwines $\phi^\natural$. Note that $\phi^\natural = \prod_{i=0}^{d_\phi} \hat\phi_i |_{(K_{\phi, r}^+)^F}$ by definition. Then the statement follows as in the first part of the proof of \cite[Theorem 3.1]{F21a}, using Lemma \ref{intertwine-i} instead of \cite[Lemma 3.4]{F21a}.
\end{proof}

\begin{proposition} \label{red-to-K}
    We have \[\<\CR_{T, U, r}^G(\phi), \CR_{T, U, r}^G(\phi)\>_{G_r^F} = \<H_c^*(Z_{\phi, U, r}^K, \ov \BQ_\ell)[\phi], H_c^*(Z_{\phi, U, r}^K, \ov \BQ_\ell)[\phi]\>_{K_{\phi, r}^F}.\]
\end{proposition}
\begin{proof}
    As $Z_{\phi, U, r} = \sqcup_{g \in G_r^F / K_{\phi, r}^F} g Z_{\phi, U, r}^K$, we have \[\CR_{T, U, r}^G(\phi) = \ind_{K_{\phi, r}^F}^{G_r^F} H_c^*(Z_{\phi, U, r}^K, \ov\BQ_\ell)[\phi].\] Then the statement follows from Lemma \ref{intertwine}.
\end{proof}

\subsection{The self inner product of $H_c^*(Z_{\phi, U, r}^K, \ov\BQ_\ell)[\phi]$} For simplicity, we set $K = K_{\phi, r}$, $H = H_{\phi, r}$, $L = L_r$, $H_U = H_{\phi, r} \cap U_r$, $L_{\ov U} = L_r \cap \ov U_r$ and so on. For a subset $R \subseteq K_{\phi, r}$ we denote by $\bar R$ the natural image of $R$ in the quotient space $K_{\phi, r} / E_{\phi, r}$.

Note that $\bar K  = \bar H \bar L = \bar L \bar H = \bar H_U \bar H_{\ov U} \bar L$, where \[\bar H_U \cong \bigoplus_{\a \in \Phi^+} \bar H^\a, \quad  \bar H_{\ov U} \cong \bigoplus_{\a \in \Phi^-} \bar H^\a\] where $\bar H^\a = (G^\a)_r^{r(\a)/2} / (G^\a)_r^{r(\a)/2+}$ and $r(\a) = r^\phi(\a)$ is as in \S \ref{subsec: tech} for $\a \in \Phi$. Moreover, $[ \bar H^\a, \bar H^\b] = 0$ if $\a \neq - \b$ and $[\bar H^\a, \bar H^{-\a}] = (T^\a)_r^{r(\a)} /  (T^\a)_r^{r(\a)^+} \subseteq \bar L$ if $\bar H^\a \neq \{0\}$.

Recall that $Z_{\phi, U, r}^K = Z_{\phi, U, r} \cap K_{\phi, r}$ and $Z_{\phi, U, r}^L = Z_{\phi, U, r} \cap L_r$. Then
\[\bar Z_{\phi, U, r}^K \cong \{g \in \bar K; g\i F(g) \in \bar K_U\}, \quad \bar Z_{\phi, U, r}^L  \cong \{g \in \bar L; g\i F(g) \in \bar L_U\}.\] By Deligne-Lusztig reduction, the isotropic spaces \begin{align*} H_c^*(\bar\Sigma_K, \ov\BQ_\ell)_{\phi, \phi\i} &= H_c^*(\bar\Sigma_K, \ov\BQ_\ell)[\phi|_{T_r^F} \boxtimes \phi|_{T_r^F}\i] \\ H_c^*(\bar\Sigma_L, \ov\BQ_\ell)_{\phi, \phi\i} &= H_c^*(\bar\Sigma_L, \ov\BQ_\ell)[\phi|_{T_r^F} \boxtimes \phi|_{T_r^F}\i] \end{align*} are virtual representations of $K_{\phi, r}^F$ and $L_r^F$ respectively.

Consider the  varieties \begin{align*} \bar\Sigma_K &= \{(x, x', y) \in F\bar K_U \times F\bar K_U \times \bar K; x F(y) = y x'\} \\ &\cong \{(x, x', u, v, \t) \in F\bar K_U \times F\bar K_U \times \bar H_U \times \bar H_{\ov U} \times \bar L; x F(v \t)= u v \t x'\}; \\ \bar\Sigma_L &= \{(x, x', y) \in F\bar L_U \times F\bar L_U \times \bar L; x F(y) = y x'\},\end{align*} on which $T_r^F \times T_r^F$ acts in the usual way as in \S \ref{subsec:sigma}. Again as in \cite[\S 6.6]{DL}, we have \begin{align*}\<H_c^*(\bar Z_{\phi, U, r}^K, \ov\BQ_\ell)[\phi], H_c^*(\bar Z_{\phi, U, r}^K, \ov\BQ_\ell)[\phi]\>_{K_{\phi, r}^F} &= \dim H_c^*(\bar\Sigma_K, \ov\BQ_\ell)_{\phi, \phi\i}; \\ \<H_c^*(\bar Z_{\phi, U, r}^L, \ov\BQ_\ell)[\phi], H_c^*(\bar Z_{\phi, U, r}^L, \ov\BQ_\ell)[\phi]\>_{L_r^F} &= \dim H_c^*(\bar\Sigma_L, \ov\BQ_\ell)_{\phi, \phi\i}.\end{align*} Here we write $H_c^*(\bar Z_{\phi, U, r}^L, \ov\BQ_\ell)[\phi] = H_c^*(\bar Z_{\phi, U, r}^L, \ov\BQ_\ell)[\phi|_{T_r^F}]$ for simplicity.
\begin{lemma} \label{prod-L}
    We have $H_c^*(\bar Z_{\phi, U, r}^L, \ov\BQ_\ell)[\phi] = \CR_{T, U \cap L, r}^L(\phi) = R_{T, U \cap L, r}^L(\phi)$. In particular, $\dim H_c^*(\bar\Sigma_L, \ov\BQ_\ell)_{\phi, \phi\i} = |\stab_{W_{L_r}(T_r)^F}(\phi|_{T_r^F})|$.
\end{lemma}
\begin{proof}
    Note that $\bar Z_{\phi, U, r}^L \cong Z_{\phi, U, r}^L / (E_{\phi, r} \cap L_r)$. By Proposition \ref{Levi-CR} we have \begin{align*} H_c^*(\bar Z_{\phi, U, r}^L, \ov\BQ_\ell)[\phi] &\cong \CR_{T, U \cap L, r}^L(\phi) \\ &\cong \CR_{T, U \cap L, r}^L(\phi_{-1}) \otimes \phi_0 |_{L_r^F} \otimes \cdots \otimes \phi_{d_\phi} |_{L_r^F} \\&\cong R_{T, U \cap L, r}^L(\phi_{-1}) \otimes \phi_0 |_{L_r^F} \otimes \cdots \otimes \phi_{d_\phi} |_{L_r^F} \\ &\cong R_{T, U \cap L, r}^L(\phi), \end{align*} where the second and the fourth isomorphisms follow from Proposition \ref{ass-CR} and Theorem \ref{ass-R} respectively, and the third one follows from that \[\CR_{T, U \cap L, r}^L(\phi_{-1}) \cong \CR_{T, U \cap L, 0}^L(\phi_{-1}) = R_{T, U \cap L, 0}^L(\phi_{-1}) \cong R_{T, U \cap L, r}^L(\phi_{-1})\] by Proposition \ref{inf-CR} and Theorem \ref{degenracy}. The second statement follows from Theorem \ref{R-R}.
\end{proof}

\begin{proposition} \label{K-K}
    We have \[\<\CR_{T, U, r}^G(\phi), \CR_{T, U, r}^G(\phi)\>_{G_r^F} = \dim H_c^*(\bar\Sigma_K, \ov\BQ_\ell)_{\phi, \phi\i} = |\stab_{W_{L_r}(T_r)^F}(\phi |_{T_r^F})|.\]
\end{proposition}
\begin{proof}
By Proposition \ref{red-to-K} and Proposition \ref{natural}, it suffices to show the second equality. There is a decomposition $\bar\Sigma_K = \bar\Sigma_K' \sqcup \bar\Sigma_K''$, where $\Sigma_K''$ is defined by the condition $v = 0$.

Note that the commutative group \[D = \{(t, t') \in T_r \times T_r; t\i F(t) = {t'}\i F(t') \in Z(L)_r^\circ\}\] acts on $\bar\Sigma''$ in the usual way. Moreover, as in the proof of Lemma \ref{red-Levi}, we have \[(\bar\Sigma_K'')^{D_\red^\circ} = \bar\Sigma_L = \{(x', x', y) \in F\bar L_U \times  F\bar L_U \times \bar L; x F(y) = y x'\}.\] By Lemma \ref{prod-L} we have \begin{align*}\dim(\bar\Sigma_K'', \ov \BQ_\ell)_{\phi, \phi\i}  &= \dim((\bar\Sigma_K'')^{D_\red^\circ}, \ov \BQ_\ell)_{\phi, \phi\i} = \dim H_c^*(\bar\Sigma_L, \ov\BQ_\ell)_{\phi, \phi\i} \\ &= |\stab_{W_{L_r}(T_r)^F}(\phi |_{T_r^F})|. \end{align*}

It remains to show $ H_c^*(\bar\Sigma_K', \ov\BQ_\ell)_{\phi, \phi\i} = 0$. Note that the action of $(T_r^{0+})^F \cong (T_r^{0+})^F \times \{1\} \subseteq (T_r^{0+})^F \times (T_r^{0+})^F$ on $\bar\Sigma_K'$ is given by \[t: (x, x', u, v, \t) \mapsto (txt\i, x', tut\i, tvt\i, t\t).\] Let $H_c^i(\bar\Sigma_K', \ov \BQ_\ell)_\phi$ be the subspace on which $(T_r^{0+})^F$ acts via $\phi$. It suffices to show $H_c^i(\bar\Sigma_K', \ov \BQ_\ell)_\phi = 0$. For $v \in \bar H_{\ov U}$ and $\a \in \Phi^-$ let $v_\a \in \bar H^\a$ be such that $v = \sum_{\a \in \Phi^-} v_\a$. We fix a total order $\le$ on $\Phi^-$. Then there is a decomposition \[\bar H_{\ov U} = \bigsqcup_{\a \in \Phi^-} \bar H_{\ov U}^{\ge \a},\] where $\bar H_{\overline U}^\a$ is defined by the condition that $v_\a \ne 0$ and $v_\b = 0$ for $\b < \a$. This induces a decomposition $\bar\Sigma_K' = \sqcup_{\a \in \Phi^-} \bar\Sigma_K^{\ge \a}$, and it suffices to show $H_c^*(\bar\Sigma_K^{\ge \a}, \ov \BQ_\ell)_\phi = 0$ for $\a \in \Phi^-$.

Let $\a \in \Phi^-$ such that $\bar H^\a \neq \{0\}$. Then $\bar H^\a \neq \{0\}$. Let $v \in \bar H_{\overline U}^{\ge \a}$ and $\xi \in \bar H^{-\a} \subseteq \bar H_U$. Then $v \xi = \t_{v, \xi} \xi v$, where $\t_{v, \xi} \in (T^\a)_r^{r(\a)}/(T^\a)_r^{r(\a)+} \subseteq \bar L$. Moreover, the map $\xi \mapsto \t_{v, \xi}$ induces an isomorphism $\bar H^{-\a} \cong (T^\a)_r^{r(\a)}/(T^\a)_r^{r(\a)+}$. Then using a similar but simpler argument as in Lemma \ref{vanish}, we deduce that $H_c^*(\bar\Sigma_K^{\ge \a}, \ov \BQ_\ell)_\phi = 0$ as desired.
\end{proof}

\begin{lemma} \label{stabilizer}
    We have \[\stab_{W_{G_r}(T_r)^F}(\phi|_{T_r^F}) = \stab_{W_{L_r}(T_r)^F}(\phi|_{T_r^F}) = \stab_{W_{L_r}(T_r)^F}(\ph_{-1}|_{T_r^F}).\]
\end{lemma}
\begin{proof}
    It is proved in  \cite[Lemma 3.6.5]{Kal}.
\end{proof}

\begin{theorem} \label{R=CR}
    We have $R_{T, U, r}^G(\phi) = \CR_{T, U, r}^G(\phi)$.
\end{theorem}
\begin{proof}
By Theorem \ref{R-R}, Theorem \ref{CR-R}, Lemma \ref{stabilizer} and Proposition \ref{K-K}, we have \[\<R_{T, U, r}^G(\phi) - \CR_{T, U, r}^G(\phi), R_{T, U, r}^G(\phi) - \CR_{T, U, r}^G(\phi)\>_{G_r^F} = 0.\] So the statements follows.
\end{proof}

\section{The representation $H_c^*(\bar Z_{\phi, U, r}^H, \BQ_\ell)[\phi]$} \label{sec: concentration}

Let notation be as in \S \ref{sec-algebraic}. Recall that $Z_{\phi, U, r}^H = Z_{\phi, U, r} \cap H_{\phi, r}$. Write $\bar H = H_{\phi, r} / E_{\phi, r}$, $\bar T^+ = T_r^{0+} / (E_{\phi, r} \cap T_r^{0+})$ and $\bar Z^H = Z_{\phi, U, r}^H / E_{\phi, r}$. We have \[\bar H \cong \bar T^+ \prod_{\a \in D} \bar H^\a,\] where $\bar H^\a = (G^\a)_r^{r(\a)/2} / (G^\a)_r^{r(\a)/2+}$ and $D = \{\a \in \Phi; \bar H^\a \neq \{0\}\}$. Here $r(\a) = r^\phi(\a)$ is as in \S \ref{subsec: tech}. Note that $D = F(D) = -D$. For each $\g \in D$ we fix an isomorphism $u_\g: \BG_a \overset \sim \longrightarrow \bar H^\g$.

\begin{lemma} \label{commutator}
    Let $\a, \b \in D$. If $\a + \b \neq 0$, we have $[\bar H^\a, \bar H^\b] = \{0\}$. Otherwise, we have \[[u_\a(x), u_\b(y)] = \a^\vee(1 + \varpi^{r(\a)} c_\a x y) \in \bar T^+,\] where $c_\a \in \ov\BF_q^\times$ is certain constant and $\a^\vee$ denotes the coroot of $\a$.
\end{lemma}

Let $C \subseteq D$ be a subset. We set $C^\pm = C \cap \Phi^\pm$. If $C \cap -C = \emptyset$, we write $\bar H^C = \prod_{\a \in C} \bar H^\a$, which is commutative subgroup of $\bar H$ isomorphic to $\BA^{|C|}$ by Lemma \ref{commutator}. By definition we have \[\bar Z^H = \{g \in \bar H; g\i F(g) \in \bar H^{F(D^+)}\}.\] Then main result of this section is
\begin{theorem} \label{one-degree}
    There exists a unique non-negative integer $n_\phi$ such that $H_c^i(\bar Z^H, \ov\BQ_\ell)[\phi] \neq \{0\}$ if and only if $i = n_\phi$.
\end{theorem}
The theorem will be proved in Corollary \ref{one-degree-flat}.

\subsection{Reductions} Let $D_s^\pm$ be the union of $F$-orbits $\CO$ of $D$ such that $\CO \subseteq \Phi^\pm$. Put $D_m = D \setminus (D_s^+ \cup D_s^-)$. Then  $D_s^+ = F(D_s^+) = -D_s^-$ and $D_m = F(D_m) = -D_m$. Let \[\bar H^\natural = \bar T^+ \prod_{\a \in D_s^+ \cup D_m} \bar H^\a,\] which is an $F$-stable subgroup. Set $\bar Z^\natural = \bar Z^\natural \cap \bar H^\natural$. As $\bar H^{F(D^+)} \subseteq \bar H^\natural$, we have \[ \bar Z^H = \cup_{g \in \bar H^F / (\bar H^\natural)^F} g \bar Z^\natural.\]

\begin{lemma} \label{cross}
    The map $(x, y) \mapsto x\i y F(x)$ gives an isomorphism  \[ \bar H^{D_m^+ \cap F(D_m^+)} \times \bar H^{D_s^+ \cup (D_m^- \cap F(D_m^+))} \overset \sim \to \bar H^{D_s^+ \cup F(D_m^+)}.\]
\end{lemma}
\begin{proof}
    The proof is straightforward by using that the subgroup $\bar H^{D_s^+ \cup F(D_m^+)}$ is commutative and that $F \bar H^\g = \bar H^{F(\g)}$ for $\g \in D$.
\end{proof}

Now we define $\bar Z^\flat = \{g \in \bar H^\natural; g\i F(g) \in \bar H^{D_s^+ \cup (D_m^- \cap F(D_m^+))}\}$.
\begin{proposition} \label{flat}
    The map $(x, z) \mapsto z x$ gives an isomorphism \[\bar H^{D_m^+ \cap F(D_m^+)} \times \bar Z^\flat \overset \sim \to \bar Z^\natural.\] In particular, $H_c^i(\bar Z^\natural, \ov\BQ_\ell)[\phi] \cong H_c^{i-2|D_m^+ \cap F(D_m^+)|}(\bar Z^\flat, \ov\BQ_\ell)[\phi]$ for $i \in \BZ$.
\end{proposition}
\begin{proof}
    It follows directly form Lemma \ref{cross}.
\end{proof}

Let $\pi: \bar H^\natural \to A := \bar H^\natural / \bar T^+ = \oplus_{\a \in D_s^+ \cup D_m} A^\a$ be the quotient map, where $A^\a \cong \bar H^\a$ is the natural image of $\bar H^\a$ in $A$. Note that $A$ is a commutative group. For $C \subseteq D_s^+ \cup D_m$ we set $A^C = \oplus_{\a \in A} A^\a$. Define \[Y = \pi(\bar Z^\flat) = \{g \in A; g\i F(g) \in  A^{D_s^+ \cup (D_m^- \cap F(D_m^+))}\}.\]

Let $\g \in D^- \cap F(D^+) \subseteq D_m$. Let $0 < a_\g < b_\g$ be the minimal positive integers such that $F^{a_\g}(\g) \in D^+ \cap F(D^-)$ and $F^{b_\g}(\g) \in D^- \cap F(D^+)$.
\begin{lemma} \label{iso}
    The natural projection $A = A^{D_s^+ \cup D_m} \to A^{D_s^+ \cup (D^- \cap F(D^+)}$ induces an isomorphism \[f: Y \overset \sim  \to A^{D_s^+ \cup (D^- \cap F(D^+)}.\] Moreover, $f\i = \pi \circ h$, where $h: A^{D_s^+ \cup (D^- \cap F(D^+)} \to \bar H^\natural$ is defined by \[(x_\g)_{\g \in D_s^+ \cup (D^- \cap F(D^+)} \mapsto \prod_{\g \in D_s^+} u_\g(x_\g) \prod_{\g \in D^- \cap F(D^+)} u_\g(x_\g) F(u_\g(x_\g)) \cdots F^{b_\g-1}(u_\g(x_\g)).\]
\end{lemma}
\begin{proof}
    As $A$ is commutative, the statement follows in the same way of \cite[Lemma 5.4]{IN24a}.
\end{proof}

We define \[\varphi = \pr \circ L \circ h \circ f: Y \to \bar T^+,\] where $h$ is as in Lemma \ref{iso}, $L: \bar H^\natural \to \bar H^\natural$ is the Lang's map given by $g \mapsto g\i F(g)$ and $\pr: \bar H^{D_s^+ \cup (D^- \cap F(D^+))} \bar T^+ \to \bar T^+$ is the natural projection.
\begin{proposition} \label{cartesian}
    There is a natural Cartesian square \[ \xymatrix{
    \bar Z^\flat \ar[d]_{\pi} \ar[r]^\d &  \bar T^+ \ar[d]^{-L} \\
    Y \ar[r]^{\varphi} &  \bar T^+,} \] where $\d$ is the projection given by \[\pi\i(Y) = h(f(Y)) \bar T^+ \cong h(f(Y)) \times \bar T^+ \to \bar T^+.\]
\end{proposition}
\begin{proof}
    By definition we have $\bar Z^\flat \subseteq \pi\i(Y)$. As $h \circ f: Y \to \bar H^\natural$ is a section of $Y$, we have $\pi\i(Y) = h(f(Y)) \bar T^+$. Then the statement follows from the definition of $\bar Z^\flat$ and  that $\bar T^+$ is the center of $\bar H^\natural$.
\end{proof}

\subsection{The local system $\CL_\phi$}
Note that $\bar T^+$ is a commutative unipotent algebraic group over $\BF_q$. By \cite[Lemma 6.1]{B}, for each character $\chi: (\bar T^+)^F \to \ov \BQ_\ell^\times$ there is a unique multiplicative local system (up to isomoprhism) $\CL_\chi$ on $\bar T^+$ whose Frobenius-trace function equals $\chi$. We write $\CL_\phi = \CL_{\phi |_{(\bar T^+)^F}}$.

The following result, due to Boyarchenko \cite[Proposition 2.10]{B} (see also \cite[Proposition 4.2.1]{Ch20}), provides an inductive way to compute the cohomology of $\CL_\phi$.
\begin{proposition} \label{iteration}
    Let $X_1$ be a variety over $\ov \BF_q$ and let $\xi: X = X_1 \times \BG_a \to \bar T^+$ be a morphism of the form \[(x, y) \mapsto \eta(x, y) \z(x)\] such that for any $x \in X_1$ the morphism $\eta_x: \BG_a \to \bar T^+$ given by $y \mapsto \eta(x, y)$ is a group homomorphism. Then we have \[H_c^i(X, \xi^* \CL_\phi) \cong H_c^i(V, (\xi|_V)^* \CL_\phi),\] where $V \subseteq X$ is the closed subvariety consisting of points $(x, y) \in X$ such that $\eta_x^* \CL_\phi$ is trivial.
\end{proposition}

We also need the following explicit computations for $\BG_a \cong \BA^1$.
\begin{proposition} \label{cohomology}
    Let $\CO$ be an $F$-orbit of $D$ and let $\g \in \CO$. Then

    (1) $H_c^i(\BG_a, \k^* \CL_\phi) = 0$ for any $i \in \BZ$, where $\k: \BG_a \to \bar T^+$ is given by $x \mapsto \g^\vee(1 + \varpi^{r(\g)} x)$;

    (2) If $|\CO|$ is even and $F^{|\CO|/2}(\g) = - \g$, then \begin{align*}\dim H_c^i(\BG_a, \t^* \CL_\phi) =  \begin{cases} q^{|\CO|/2}, & \text{ if } i=1; \\ 0, &\text{ otherwise,} \end{cases}\end{align*} where $\t: \BG_a \to \bar T^+$ is given by $x \mapsto \g^\vee(1 + \varpi^{r(\g)} x^{q^{|\CO|/2} + 1})$.
\end{proposition}
\begin{proof}
    The first statement follows form Proposition \ref{iteration}. The second follows from \cite[Proposition 5.16]{IN24a}, which is based on \cite[Proposition 6.6.1]{BW}.
\end{proof}

\subsection{The computation}
Let $\CO$ be an $F$-orbit of $D_m$. We set $\hat\CO = \CO \cup -\CO$. Fix a subset $\{\CO_1, \dots, \CO_{n_0}\}$ of $F$-orbits of $D_m$ such that $D_m$ is a disjoint union of $\hat \CO_i$ for $1 \le i \le n_0$.
\begin{lemma} \label{product}
  Let  $\varphi$ be as in Proposition \ref{cartesian}. Then we have
  \begin{itemize}
      \item $H_c^i(\bar Z^\flat, \ov \BQ_\ell)[\phi] \cong H_c^i(Y, \varphi^* \CL_\phi)$;

      \item $Y \cong (Y \cap A^{D_s^+}) \times (Y \cap A^{\hat\CO_1}) \times \cdots \times (Y \cap A^{\hat \CO_{n_0}})$;

      \item $\varphi^*\CL_\phi \cong \varphi_{D_s^+}^*\CL_\phi \boxtimes \varphi_{\hat \CO_1}^*(\CL_\phi) \boxtimes \cdots  \boxtimes \varphi_{\hat \CO_{n_0}}^*(\CL_\phi)$.
  \end{itemize}
  Here $\varphi_C$ denotes the restriction of $\varphi$ to $Y \cap A^C$ with $C = D_s^+$ or $C = \hat \CO_i$ for $1 \le i \le n_0$.
\end{lemma}
\begin{proof}
    By Proposition \ref{cartesian} the projection $\pi|_{\bar Z^\flat}: \bar Z^\flat \to Y$ is a $(\bar T^+)^F$-torsor. Hence we have $H_c^i(\bar Z^\flat, \ov \BQ_\ell)[\phi] \cong H_c^i(Y, \varphi^* \CL_\phi)$ and the first statement follows. The last two statements follow by observing that the subgroups $\pi\i(A^C) \subseteq \bar H^\natural$ for $C = D_s^+, \hat \CO_1, \dots, \hat \CO_{n_0}$ are $F$-stable and commute with each other.
\end{proof}

\begin{lemma} \label{one-degree-part}
    Let $C = D_s^+$ or $C = \hat \CO$ for some $F$-orbit $\CO \subseteq D_m$. Then there is a unique non-negative integer $n_C$ such that $H_c^i(Y \cap A^C, \varphi_C^* \CL_\phi) \neq \{0\}$ if and only if $i = n_C$.
\end{lemma}
\begin{proof}
    First assume that $C = D_s^+$. By Lemma \ref{iso} we have $Y \cap A^C = A^C \cong \BA^{|D_s^+|}$. Moreover, $\varphi_C$ is the trivial map since $\pi\i(A^C)$ is commutative. Hence $H_c^i(Y \cap A^C, \varphi_C^* \CL_\phi) \cong H_c^i(\BA^{|D_s^+|}, \ov\BQ_\ell)$ and the statement follows.

    Now we assume that $C = \hat \CO = \CO \cup -\CO$, where $\CO$ is an $F$-orbit of $D_m$. Then $Y \cap A^C \cong A^{C^- \cap F(C^+)}$ by Lemma \ref{iso}. Choose $\g \in \CO^- \cap F(\CO^+)$. Then there is sequence of integers \[0 = a_0 < a_1 < a_2 < \cdots < a_{2m_0} = |\CO|\] such that $F^{a_{2i-1}}(\g) \in \CO^+ \cap F(\CO^-)$ and $F^{a_{2i}}(\g) \in \CO^- \cap F(\CO^+)$  for $1 \le i \le m_0$. Set $\g_j = F^{a_j}(\g)$ for $0 \le j \le 2m_0$. We need to consider the following two cases.

    Case (1): $\CO \cap -\CO = \emptyset$. Then \[C^- \cap F(C^+) = \{\g_{2i}; 1\le i \le m_0\} \cup \{-\g_{2i-1}; 1\le i \le m_0\}.\] The map $\varphi_C$ is given by \[(x_{\g_{2i}}, x_{-\g_{2i-1}})_{1 \le i \le m_0} \mapsto \sum_{i=1}^{m_0} \g_{2i}^\vee(1 + \varpi^{r(\g_{2i})} c_{\g_{2i}} (x_{\g_{2i}} - x_{\g_{2i-2}}^{q^{a_{2i}- a_{2i-2}}}) x_{-\g_{2i-1}}^{q^{a_{2i} - a_{2i-1}}}),\] where each $c_{\g_{2i}} \in \ov \BF_q^\times$ is certain constant. Let $V \subseteq A^{C^- \cap F(C^+)} \cong Y \cap A^C$ be the closed subset defined by the equations $x_{\g_{2i}} - x_{\g_{2i-2}}^{q^{a_{2i}- a_{2i-2}}} = 0$ for $1 \le i \le m_0$. Then $V$ is a disjoint union of $q^{|\CO|}$ copies of $\BA^{m_0}$ and the restriction of $\varphi_C$ to $V$ is trivial. Applying Proposition \ref{iteration} repeatedly, we have \[H_c^i(Y \cap A^C, \varphi_C^* \CL_\phi) \cong H_c^i(V, \ov\BQ_\ell) \cong H_c^i(\BA^{m_0}, \ov\BQ_\ell)^{\oplus q^{|\CO|}}.\] Hence the statement holds in this case.

    Case (2): $\CO = -\CO$. Then $C^- \cap F(C^+) = \{\g_{2i}; 1 \le i \le m_0\}$, $|\CO| = 2 a_{m_0}$, $m_0$ is odd and $\g_{i+m_0} = -\g_i$ for $0 \le i \le m_0$. The map $\varphi_C$ is given by \begin{align*} (x_{\g_{2i}})_{1 \le i \le m_0} &\mapsto \sum_{i=1}^{(m_0-1)/2} \g_{2i}^\vee(1 + \varpi^{r(\g_{2i})} c_{\g_{2i}} (x_{\g_{2i}} - x_{\g_{2i-2}}^{q^{a_{2i} - a_{2i-2}}}) x_{\g_{2i-1+m_0}}^{q^{a_{2i+m_0} - a_{2i-1+m_0}}}) \\ &\quad\quad + \g_{2m_0}^\vee(1 + \varpi^{r(\g_{2m_0})} c_{\g_{2m_0}} x_{\g_{2m_0}} x_{\g_{m_0-1}}^{q^{a_{m_0} - a_{m_0-1}}}), \end{align*} where each $c_{\g_{2i}} \in \ov \BF_q^\times$ is certain constant. Let $V \subseteq A^{C^- \cap F(C^+)} \cong Y \cap A^C$ be the closed subset defined by the equations $x_{\g_{2i}} - x_{\g_{2i-2}}^{q^{a_{2i}- a_{2i-2}}} = 0$ for $1 \le i \le (m_0 - 1)/2$. Then $V \cong \BA^{(m_0-1)/2} \times \BA^1$ and the restriction of $\varphi_C$ to $V$ is the composition of the natural projection $\BA^{(m_0-1)/2} \times \BA^1 \to \BA^1$ with the following morphism \[\t: \BA^1 \to \bar T^+, \quad x \mapsto \g^\vee(1 + \varpi^{r(\g)} c_\g x^{1 + q^{|\CO|/2}}).\] Applying Proposition \ref{iteration} repeatedly, we have \[H_c^i(Y \cap A^C, \varphi_C^* \CL_\phi) \cong H_c^i(V, (\varphi_C^* \CL_\phi)|_V) \cong H_c^{i - m_0 + 1}(\BA^1, \t^* \CL_\phi ).\] Hence the statement also holds by Proposition \ref{cohomology} (2).
\end{proof}

\begin{corollary} \label{one-degree-flat}
    There is a non-negative integer $n_\phi^\flat$ such that $H_c^i(\bar Z^\flat, \ov\BQ_\ell)[\phi] \neq \{0\}$ if and only if $i = n_\phi^\flat$. As a consequence, Theorem \ref{one-degree} is true.
\end{corollary}
\begin{proof}
    By Lemma \ref{product} and the K\"{u}nneth formula, we have \begin{align*} &\quad\ H_c^i(\bar Z^\flat, \ov\BQ_\ell)[\phi] \\ &\cong \bigoplus_{(i_j)_{0 \le j \le n_0}, \sum_j i_j = i} H_c^{i_0}(Y \cap A^{D_s^+}, \varphi_{D_s^+}^* \CL_\phi) \otimes H_c^{i_1}(Y \cap A^{\hat \CO_1}, \varphi_{\hat \CO_1}^* \CL_\phi)  \otimes \cdots \otimes H_c^{i_{n_0}}(Y \cap A^{\hat \CO_{n_0}}, \varphi_{\hat \CO_{n_0}}^* \CL_\phi). \end{align*} Hence the first statement follows from Lemma \ref{one-degree-part}. The second statement follows from the decomposition $\bar Z^H = \cup_{g \in \bar H^F / (\bar H^\natural)^F} g \bar Z^\natural$ and Proposition \ref{flat}.
\end{proof}

\section{An irreducible decomposition} \label{sec-decomp}
Let notation be as in \S \ref{sec-algebraic}, and assume moreover that $(G^i)_{0 \le i \le d_\phi}$ is standard with respect to $U$. In this section we give an explicit description of $R_{T, U, r}^G(\phi) = \CR_{T, U, r}^G(\phi)$.

\subsection{The $K_{\phi, r}^F$-module $\k_\phi$} \label{subsec:kappa}
Let $Z_{\phi, U, r}^H = Z_{\phi, U, r} \cap H_{\phi, r}$ and $\bar Z_{\phi, U, r}^H = Z_{\phi, U, r}^H / E_{\phi, r}$. Note that the finite groups $(T_r^{0+})^F$ and $(K_{\phi, r}^+)^F = E_{\phi, r}^F (T_r^{0+})^F$ act on $\bar Z_{\phi, U, r}^H$ by right multiplication. Hence the corresponding isotropic subspaces \[H_c^*(\bar Z_{\phi, U, r}^H, \ov\BQ_\ell)[\phi] := H_c^*(\bar Z_{\phi, U, r}^H, \ov\BQ_\ell)[\phi |_{(T_r^{0+})^F}] \text{ and } H_c^*(\bar Z_{\phi, U, r}^H, \ov\BQ_\ell)[\phi^\natural]\] are virtual representations of $H_{\phi, r}^F$. Here $\phi^\natural$ is the character of $(K_{\phi, r}^+)^F$ defined in \S\ref{subsec:CR}.
\begin{lemma} \label{coincidence}
    We have $H_c^i(\bar Z_{\phi, U, r}^H, \ov\BQ_\ell)[\phi] = H_c^i(\bar Z_{\phi, U, r}^H, \ov\BQ_\ell)[\phi^\natural]$ for $i \in \BZ$.
\end{lemma}
\begin{proof}
    As $\phi^\natural |_{(T_r^{0+})^F} = \phi |_{(T_r^{0+})^F}$, it suffices to show the left hand side is contained in the right hand side. Let $g \in (K_{\phi, r}^+)^F$. Then $g = h t$ for some $t \in  (T_r^{0+})^F$ and $h \in E_{\phi, r}^F$. As $E_{\phi, r}^F$ acts on $\bar Z_{\phi, U, r}^H$ trivially, $g = h t$ acts $H_c^i(\bar Z_{\phi, U, r}^H, \ov\BQ_\ell)[\phi]$ by the scalar $\phi(t) = \phi^\natural(g)$. So the statement follows.
\end{proof}

\begin{proposition} \label{Heisenberg-group}
     If $d_\phi = 0$ and $\phi_{d_\phi} = 1$, then $H_{\phi, r}^F / \ker \phi^\natural$ is trivial. Otherwise, it is a Heisenberg $p$-group with center $(K_{\phi, r}^+)^F / \ker \phi^\natural$.
\end{proposition}
\begin{proof}
    If $d_\phi = 0$ and $\phi_{d_\phi} = 1$, then $H_{\phi, r} = G_r^{0+}$, $\phi^\natural = 1$ and hence $H_{\phi, r}^F / \ker \phi^\natural = \{1\}$. Otherwise, the statement is proved in \cite[Proposition 18.1]{Kim}.
\end{proof}

\begin{proposition} \label{Heisenberg}
    The $H_{\phi, r}^F$-module $\pm H_c^*(\bar Z^H_{\phi, U, r}, \ov \BQ_\ell)[\phi]$ is irreducible. In particular, when $H_{\phi, r}^F / \ker \phi^\natural$ is nontrivial, it is the inflation of the unique irreducible Heisenberg representation with a nontrivial central character $\phi^\natural |_{(K_{\phi, r}^+)^F / \ker \phi^\natural}$.
\end{proposition}
\begin{proof}
By a similar but simpler argument in Proposition \ref{K-K}, we have \[\<H_c^*(\bar Z^H_{\phi, U, r}, \ov\BQ_\ell)[\phi], H_c^*(\bar Z^H_{\phi, U, r}, \overline \BQ_\ell)[\phi]\>_{H_{\phi, r}^F} = 1.\] Hence $\pm H_c^*(\bar Z^H_{\phi, U, r}, \ov \BQ_\ell)[\phi]$ is an irreducible $H_{\phi, r}^F$-module. By Proposition \ref{natural}, $(K_{\phi,r}^+)^F$ acts on $H_c^*(\bar Z^H_{\phi, U, r}, \ov\BQ_\ell)[\phi]$ by the character $\phi^\natural$. Assume $H_{\phi, r}^F / \ker \phi^\natural$ is non-trivial, then it is a Heisenberg $p$-group with center $(K_{\phi, r}^+)^F / \ker \phi^\natural$ by Proposition \ref{Heisenberg-group}. Hence $\pm H_c^*(\bar Z^H_{\phi, U, r}, \ov \BQ_\ell)[\phi]$ is an irreducible $H_{\phi, r}^F / \ker \phi^\natural$-module with a non-trivial central character $\phi^\natural |_{(K_{\phi, r}^+)^F / \ker \phi^\natural}$, which is uniquely determined by the representation theory of Heisenberg $p$-groups.
\end{proof}

\smallskip

We write $L = G^0$. Since $(G^i)_{-1 \le i \le d_\phi}$ is standard with respect to $U$, $L_r$ normalizes $H_{\phi, r} \cap \CI_{\phi, U, r}$. Moreover, we have $[L_r, T_r^{0+}] \subseteq E_{\phi, r}$. Hence $L_r^F \times (T_r^{0+})^F$ acts on $\bar Z^H_{\phi, U, r}$ by \[(y, t): z \mapsto y z y\i t = y z t y\i.\] This induces an action $\ad_{L_r^F}$ of $L_r^F$ on $H_c^i(\bar Z^H_{\phi, U, r}, \ov\BQ_\ell)[\phi]$ for $i \in \BZ$. Since $[L_r^{0+}, H_{\phi, r}] \subseteq E_{\phi, r}$, the action $\ad_{L_r^F}$ factors through the quotient $L_0^F$. By Proposition \ref{natural} and that $K_{\phi, r}^F = H_{\phi, r}^F L_r^F$, each $H_{\phi, r}^F$-module $H_c^i(\bar Z^H_{\phi, U, r}, \ov\BQ_\ell)[\phi]$ extends to a $K_{\phi, r}^F$-module on which $L_r^F$ acts by $\ad_{L_r^F}$ times the character $\prod_{i=0}^{d_\phi} \phi_i|_{L_r^F}$. Thus we can define a virtual $K_{\phi, r}^F$-module \[\k_\phi = \k_{\phi, U} := \sum_i (-1)^i H_c^i(\bar Z^H_{\phi, U, r}, \ov\BQ_\ell)[\phi].\] By Theorem \ref{one-degree}, $\pm \k_\phi$ is a genuine $K_{\phi, r}^F$-module, whose restriction to $H_{\phi, r}^F$ is the irreducible $H_{\phi, r}^F$-module as in Proposition \ref{Heisenberg}.

\subsection{The irreducible decomposition}
Now we describe the virtual $K_{\phi, r}^F$-module $H_c^*(Z_{\phi, U, r}^K, \ov\BQ_\ell)[\phi]$. Let $(h, l, t_1, t_2) \in H_{\phi, r}^F \times L_r^F \times (T_r^{0+})^F \times T_r^F$. As $[L_r, T_r^{0+}] \subseteq E_{\phi, r}$, the map \[(h, l, t_1, t_2): (x, y) \mapsto (h l x l\i t_1, l y t_2)\] induces an action of $(H_{\phi, r}^F \rtimes L_r^F) \times (T_r^{0+})^F \times T_r^F$ on $\bar Z_{\phi, U, r}^H \times \bar Z_{\phi, U, r}^L$.
\begin{proposition} \label{tensor}
    The map $(x, y) \to xy$ induces an $(H_{\phi, r}^F \rtimes L_r^F) \times (T_r^{0+})^F \times T_r^F$-equivariant $(T_r^{0+})^F / (E_{\phi, r} \cap T_r)^F$-torsor \[f: \bar Z_{\phi, U, r}^H \times \bar Z_{\phi, U, r}^L \to \bar Z_{\phi, U, r}^K,\] In particular, there is an isomorphism of $K_{\phi, r}^F$-modules \[ H_c^*(Z_{\phi, U, r}^K, \ov \BQ_\ell)[\phi] \cong \k_\phi \otimes R_{T, U, 0}^L(\phi_{-1}).\] Here $R_{T, U, 0}^L(\phi_{-1})$ is viewed as a virtual $K_{\phi, r}^F$-module by the natural projections $K_{\phi, r}^F = H_{\phi, r}^F L_r^F \to L_r^F \to L_0^F$.
\end{proposition}
\begin{proof}
First note that $f$ is well-defined and is $(H_{\phi, r}^F \rtimes L_r^F) \times (T_r^{0+})^F \times T_r^F$-equivariant, since $[L_r, T_r^{0+}] \subseteq E_{\phi, r}$ and $H_{\phi, r} \cap F\CI_{\phi, U, r}$ is normalized by $L_r$.

Let $z \in Z_{\phi, U, r}^K$. Write $z = x y$ with $x \in H_{\phi, r}$ and $y \in L_r$. Then \[(y\i x\i F(x) y) y\i F(y) = z\i F(z) \in F\CI_{\phi, U, r} = (H_{\phi, r} \cap F\CI_{\phi, U, r}) (L_r \cap F\CI_{\phi, U, r}).\] Hence $z\i F(z) = a b$ for some $a \in H_{\phi, r} \cap F\CI_{\phi, U, r}$ and $b \in L_r \cap F\CI_{\phi, U, r}$. Then \[y\i F(y) b\i = (y\i x\i F(x) y)\i a \in L_r \cap H_{\phi, r} = L_r^{0+}.\] By Lang's theorem, there exists $\d \in L_r^{0+}$ such that $(\d y)\i F(\d y) = b \in F\CI_{\phi, U, r}$. Thus, by replacing the pair $(x, y)$ with $(x \d\i, \d y)$, we can assume further that $y\i F(y) \in F\CI_{\phi, U, r}$, that is $y \in Z_{\phi, U, r}^L$. This implies that $y\i x\i F(x) y \in H_{\phi, r} \cap F\CI_{\phi, U, r}$ and hence $x\i F(x) \in F\CI_{\phi, U, r}$. Thus $x \in Z_{\phi, U, r}^H$ and $f$ is surjective.

Let $x, x' \in Z_{\phi, U, r}^H$ and $y, y' \in Z_{\phi, U, r}^L$ such that $x y E_{\phi, r} = x' y' E_{\phi, r}$. As $E_{\phi, r} \subseteq H_{\phi, r}$, we may assume that $x y = x' y'$. Then ${x'}\i x = y' y\i \in H_{\phi, r} \cap L_r = L_r^{0+}$. Hence we may assume further $x \in x' (L_\der)_r^{0+} t$ for some $t \in T_r^{0+}$. As $x, x' \in Z_{\phi, U, r}^H$, we have $t \in (T_r^{0+})^F (E_{\phi, r} \cap T_r)$, which implies that $x' E_{\phi, r} = x t\i E_{\phi, r} = t\i x E_{\phi, r}$ and $y' E_{\phi, r} = t y E_{\phi, r} = y t E_{\phi, r}$. Therefore, $f$ is a $(T_r^{0+})^F / (E_{\phi, r} \cap T_r)^F$-torsor as desired.

Let $\phi^\flat = \phi \circ m$, where $m: (T_r^{0+})^F \times T_r^F \to T_r^F$ is given by $(t_1, t_2) \mapsto t_1 t_2$. Then we have \begin{align*}
    &\quad\ H_c^*(Z_{\phi, U, r}^K, \ov\BQ_\ell)[\phi] \\ &\cong H_c^*(\bar Z_{\phi, U, r}^K, \ov\BQ_\ell)[\phi^\flat] \\ &\cong H_c^*(\bar Z_{\phi, U, r}^H \times \bar Z_{\phi, U, r}^L, \ov\BQ_\ell)^{(T_r^{0+})^F}[\phi^\flat] \\ &\cong H_c^*(\bar Z_{\phi, U, r}^H \times \bar Z_{\phi, U, r}^L, \ov\BQ_\ell)[\phi^\flat] \\ &\cong  H_c^*(\bar Z_{\phi, U, r}^H, \ov\BQ_\ell)[\phi] \otimes  H_c^*(\bar Z_{\phi, U, r}^L, \ov\BQ_\ell)[\phi] \\ &\cong H_c^*(\bar Z_{\phi, U, r}^H, \ov\BQ_\ell)[\phi] \otimes (\phi_0|_{L_r^F} \otimes \cdots \otimes \phi_{d_\phi}|_{L_r^F} \otimes  H_c^*(\bar Z_{\phi, U, r}^L, \ov\BQ_\ell)[\phi_{-1}]) \\ &\cong \k_\phi \otimes \CR_{T, U, r}^L(\phi_{-1}) \\ &\cong \k_\phi \otimes R_{T, U, 0}^L(\phi_{-1}),
\end{align*} where the second isomorphism follows from the first statement; the third one follows from the $\phi^\flat$ is trivial over $(T_r^{0+})^F \cong \{(t, t\i); x \in (T_r^{0+})^F\}$; the fifth one follows from Proposition \ref{ass-CR}; the last one follows from Proposition \ref{inf-CR}.
\end{proof}

\begin{theorem} \label{decomposition}
    There is an irreducible decomposition \[\CR_{T, U, r}^G(\phi) = \ind_{K_{\phi, r}^F}^{G_r^F} \k_\phi \otimes R_{T, U, 0}^L(\phi_{-1}) = \sum_\rho m_\rho \ind_{K_{\phi, r}^F}^{G_r^F} \k_\phi \otimes \rho,\] where $\rho$ ranges over irreducible $L_0^F$-modules with $m_\rho = \<\rho, R_{T, U, 0}^L(\phi_{-1})\>_{L_0^F}$.
\end{theorem}
\begin{proof}
The equalities follows from the decomposition $Z_{\phi, U, r} = \cup_{G_r^F / K_{\phi, r}^F} Z_{\phi, U, r}^K$ and Proposition \ref{tensor}. It remains to show  $\ind_{K_{\phi, r}^F}^{G_r^F} \k_\phi \otimes \rho$ are pairwise non-isomorphic irreducible $G_r^F$-modules. By Lemma \ref{intertwine}, it suffices to show $\k_\phi \otimes \rho$ are pairwise non-isomorphic irreducible $K_{\phi, r}^F$-modules. As $\k_\phi$ is an irreducible $H_{\phi, r}^F$-module and the action of $H_{\phi, r}^F$ on $\rho$ is trivial, any nontrivial $K_{\phi, r}^F$-submodule of $\k_\phi \otimes \rho$ is of the form $\k_\phi \otimes \t$, where $\t$ is some nontrivial $K_{\phi, r}^F$-submodule of $\rho$. In view that $\rho$ is an irreducible $L_r^F$-module, it follows that $\k_\phi \otimes \rho$ is an irreducible $K_{\phi, r}^F$-module.

Suppose that $\k_\phi \otimes \rho \cong \k_\phi \otimes \rho'$ for some $\rho \ncong \rho'$, that is, for any $x \in H_{\phi, r}^F$ and $y \in L_r^F$ we have \[\tr_{\k_\phi}(x y) \tr_\rho(y) = \tr_{\k_\phi}(x y) \tr_{\rho'}(y).\] Choose $y_0 \in L_r^F$ such that $\tr_\rho(y_0) \neq \tr_{\rho'}(y_0)$. Thus $\tr_{\k_\phi}(x y_0) = 0$ for any $x \in H_{\phi, r}^F$. By Burnside's theorem, the linear endomorphisms $\k_\phi(x)$ for $x \in H_{\phi, r}^F$ span the endomorphism ring of the underlying linear space of $\k_\phi$. This means that $\k_\phi(y_0)$ is the trivial endomorphism, which is impossible. The proof is finished.
\end{proof}

\section{Application on supercuspidal representations} \label{sec-app}

In this section, we show that a large class of irreducible supercuspidal representations of the $p$-adic group $G^F$ can be realized through higher Deligne-Lusztig representations $R_{T, U, r}^G(\phi)$.

\subsection{Yu's construction} \label{subsec: cusp-datum} Recall that a generic cuspidal datum of $G$ is a triple
\[ \Xi = (\L, \tx, \rho),\] where
\begin{itemize}
    \item $\L = (G^i, \phi_i, r_i)_{0 \le i \le d}$ is a generic datum as in \S \ref{subsec:Howe} such that $Z(G^0)/ Z(G)$ is anisotropic;

    \item $\tx \in \CB(G^0, k) \subseteq \CB(G, k)$, whose image $\bar \tx$ in $\CB(G_\der^0)$ is a vertex;

    \item $\rho$ is an irreducible $(G^0)_{\bar{\tx}}^F$-module such that $\rho|_{(G^0)_\tx^F}$ contains the inflation of a cuspidal representation of the reductive quotient of $(G^0)_\tx^F$.
\end{itemize}
Here $(G^0)_{\bar\tx} \supseteq (G^0)_\tx$ denotes the stabilizer of $\bar\tx$ in $G^0$. Moreover, we say $\Xi$ is normalized if $\L$ is normalized as in \S \ref{subsec:Howe}.

\smallskip

In \cite{Yu}, Yu constructed an irreducible supercuspidal representation $\pi_\Xi$ for each generic cuspidal datum $\Xi$. In \cite{HM}, Hakim and Murnaghan introduced the notion of $G$-equivalence relation on the set of generic cuspidal data, and proved, under certain assumptions, that any two generic cuspidal data $\Xi$, $\Xi'$ are $G$-equivalent if and only if $\pi_\Xi \cong \pi_{\Xi'}$. Under the assumption (*) on $p$, Kaletha dropped the additional assumptions in the previous result of Hakim and Murnaghan, and proved that each generic cuspidal datum is $G$-equivalent to a normalised one. As a result, we only need to consider normalised generic cuspidal data.

\

Now we fix a normalized generic cuspidal datum $\Xi = (\L, \tx, \rho)$ with $\L = (G^i, \phi_i, r_i)_{0 \le i \le d}$ as above. We assume further that $\Xi = ((G^i, \phi_i, r_i)_{0 \le i \le d}, \tx, \rho)$ is unramified, that is, $G^0$ (and hence any $G^i$) splits over $\brk$. Set $L = G^0$. By the unramified assumption, there exists an unramified maximal torus $T$ of $L$ such that $\tx \in \CA(T, \brk)$ and $T$ contains a maximal $k$-split torus of $L$, see \cite[pp. 585-586]{Yu}.


Let $K = K_\L$, $K^+ = K_\L^+$, $H = H_\L$ and $\chi = \chi_\L$ be defined in \S\ref{subsec:Howe} with respect to the generic datum $\L$. Set $\tilde K = H L_{\bar{\tx}}$.

Let $\k(\L)$ denote the induced Weil-Heisenberg representation of $\tilde K^F$, that is, $\k(\L)) |_{H^F}$ is inflated from the unique Heisenberg representation of $H^F / \ker\chi$ with central character $\chi |_{(K^+)^F / \ker\chi}$, and moreover, $L_{\bar\tx}^F$ acts on $\k(\L)$ by the character $\prod_{i=0}^d \phi_i|_{L_{\bar \tx}^F}$ times the induced Weil representation. We refer to \cite[\S 2.5]{F21a} for the precise construction of $\k(\L)$. We also view $\rho$ as a $\tilde K^F$-module on which $H^F$-acts trivially. The following result is proved in \cite{Yu} and \cite{F21a}.
\begin{theorem}
    The compact induction \[\pi_\Xi := \text{c-}\ind_{\tilde K^F}^{G^F} \k(\L) \otimes \rho,\] is an irreducible supercuspidal representation of $G^F$.
\end{theorem}

\subsection{The representation $\rho_0^\flat$} Put $Z = Z(G)$. Let $\rho_0$ be an irreducible representation of $Z^F L_\tx^F$ which appears in $\rho |_{Z^F L_\tx^F}$. In other words, $\rho$ is a direct summand of $\ind_{Z^F L_\tx^F}^{L_{\bar\tx}^F} \rho_0$. By assumption on $\rho$, the restriction $\rho_0 |_{L_\tx}$ is inflated from a cuspidal representation of $L_0^F$. Recall that $L_0$ denotes the reductive quotient of $L_\tx$. Let $\o = \rho |_{Z^F} = \rho_0 |_{Z^F}$ be the central character of $\rho$ or $\rho_0$.
\begin{lemma} \label{summand}
    The representation $\k(\L) \otimes \rho$ of $\tilde K^F$ is a direct summand of $\ind_{Z^F K^F}^{\tilde K^F} \k(\L) \otimes \rho_0$. Hence $\pi_\Xi$ is a direct summand of $\text{c-}\ind_{Z^F K^F}^{G^F} \k(\L) \otimes \rho_0$.
\end{lemma}
\begin{proof}
    It follows from the observation that $\k(\L) \otimes \rho_0$ is a $Z^F K^F$-submodule of $\k(\L) \otimes \rho$.
\end{proof}

Let $B = TU$ be a Borel subgroup such that $(G^i)_{0 \le i \le d}$ is standard with respect to $U$. Let $\phi = \prod_{i=0}^d \phi_i |_{T^F}$ and $\k_\phi= \k_{\phi, U}$ be the $K^F$-module constructed in \S\ref{subsec:kappa}. Moreover, we view $\k_\phi$ as a $Z^F K^F$-module on which $Z^F$ acts via the character $\phi$. Note that $(\L, 1)$ is a Howe factorization of $\phi$. Note that $\phi^\natural  = \chi$. Hence $\k_\phi |_{H^F} \cong \k(\L) |_{H^F}$ by Proposition \ref{Heisenberg}.

The following lemma is inspired from \cite[Proposition 18.5]{Kim}
\begin{lemma} \label{iso}
    There exists an irreducible module $\rho_0^\flat$ of $Z^F K^F$ such that \[ \k(\L) \otimes \rho_0 \cong \k_\phi \otimes \rho_0^\flat \] as $Z^F K^F$-modules. Moreover, $\rho_0^\flat |_{Z^F} = \o$ and the restriction of $\rho_0^\flat$ to $H^F$ is trivial.
\end{lemma}
\begin{proof}
Let $\k$ be the irreducible $Z^F H^F$-module such that $\k|_{Z^F} = \phi|_{Z^F} \o$ and $\k |_{H^F} = \k(\L) |_{H^F} = \k_\phi |_{H^F}$. By construction, $\k$ appears in the $Z^F H^F$-module $\k(\L) \otimes \rho_0$. As $\k(\L) \otimes \rho$ is an irreducible $Z^F K^F$-module, This means that \[\k(\L) \otimes \rho_0 \subseteq \th := \ind_{Z^F H^F}^{Z^F K^F} \k.\] On the other hand, consider \[\vartheta :=  \k_\phi \otimes \ind_{Z^F H^F}^{Z^F K^F} \hat\o,\] where $\hat\o$ is the representation given by $\hat\o(z h) = \o(z)$ for $z \in Z^F$ and $h \in H^F$. As $H^F$ is a normal subgroup of $Z^F K^F$, it suffices to prove $\th = \vartheta$. To this end, we show that their traces coincide. First note that $Z^F H^F$ is a normal subgroup of $Z^F K^F$. Hence $\th(g) = \vartheta(g) = 0$ if $g \in Z^F K^F \setminus Z^F H^F$. On the other hand, for $g = z h \in Z^F H^F$ with $z \in Z^F$ and $h \in H^F$, we have \begin{align*} \tr_\th(g) &=  \sum_{x \in Z^F K^F / Z^F H^F} \phi(z) \o(z) \tr_{\k_\phi}(x\i h x); \\ \tr_\vartheta(g) &=  \sum_{x \in Z^F K^F / Z^F H^F} \phi(z) \o(z) \tr_{\k_\phi}(h).\end{align*} Hence it suffices to show $\k_\phi |_{H^F} \cong ({}^x \k_\phi) |_{H^F}$ for any $x \in K^F$. Indeed, as $[K^F, (K^+)^F] \subseteq \ker\phi^\natural$, the irreducible representations $\k_\phi |_{H^F}$ and $({}^x \k_\phi) |_{H^F}$ has the same central character. By the uniqueness of Heisenberg representations, we have $\k_\phi |_{H^F} \cong ({}^x \k_\phi) |_{H^F}$ as desired.
\end{proof}

\subsection{Cuspidality of $\rho_0^\flat$} For any reductive group over a field we denote by $r_{ss}(M)$ the (absolute) semisimple rank of $M$.
\begin{proposition} \label{cuspidal}
    Assume that $q > r_{ss}(G) + 1$. Then $\rho_0^\flat |_{L_\tx^F}$ is inflated from an irreducible cuspidal representation of $L_0^F$.
\end{proposition}

Before proving this proposition, we need some preparations. We write $\sfT = T_0$ and $\sfL = L_0$. By the choice of $T$, $\sfT$ is a maximally split maximal torus of $\sfL$. Hence there is a $F$-stable Borel subgroup $\sfB \subseteq \sfL$ containing $\sfT$. Let $\Psi_\sfL$ be the root system of $\sfT$ in $\sfL$. Denote by $\Psi_\sfL^+$ the set of (positive) roots appearing in $\sfB$. Let $\D_\sfL \subseteq \Psi_\sfL^+$ be the set of simple roots. For $J \subseteq \D$ let $\Psi_J \subseteq \Psi$ be the root subsystem spanned by $J$.

Let $\sfB \subseteq \sfP \subsetneq \sfL$ be a maximal proper standard parabolic subgroup with unipotent radical $\sfN$. Denote by $\Psi_\sfP$ and $\Psi_\sfN$ the sets of roots appearing $\sfP$ and $\sfN$ respectively. Let $\L_\sfP$ be the set of $F$-fixed cocharacters $\l \in X_*(T)$ such that  \[\Psi_\sfP = \{\a \in \Psi_\sfL; \eta(\a) \ge 0\}, \text{ and hence } \Psi_\sfN = \{\a \in \Psi_\sfL; \eta(\a) > 0\}.\]

Now we construct a particular element of $\L_\sfP$. First note that there is a unique $F$-orbit $\CO_\sfP$ of $\D_\sfL$ such that $\Psi_\sfP = \Psi^+ \cup \Psi_{\D \setminus \CO_\sfP}$.
Set \[\l_\sfP = \sum_{\a \in \CO_\sfP} \o_{\a, \sfL}^\vee \in X_*(T)_\BQ,\] where $\o_{\a, \sfL}^\vee$ denotes the fundamental coweight corresponding to the simple root $\a \in \D_\sfL$. Let $n_\sfP \in \BZ_{\ge 1}$ be the smallest positive integer such that $\eta_\sfP := n_\sfP \l_\sfP^\vee$ lies in the coroot lattice of $\Psi$. It is clear that $\eta_\sfP \in \L_\sfP$.

\begin{lemma} \label{max}
    If $\Psi_\sfL$ is of type $A_{r_{ss}(\sfL)}$ and $F$ acts trivially on $\Psi_\sfL$, then we have $\max_{\g \in \Psi_\sfN}\eta_\sfP(\g) = r_{ss}(\sfL) + 1$. Otherwise, $\max_{\g \in \Psi_\sfN}\eta_\sfP(\g) \le r_{ss}(\sfL)$.
\end{lemma}
\begin{proof}
    It follows directly by a case-by-case computation.
\end{proof}

For $\eta \in \L_\sfP$ we denote by $\G_\eta \subseteq \sfL$ be the subgroup generated by $\sfN$ and the one parameter subgroup $\eta: \BG_m \to \sfL$.
\begin{lemma} \label{radical}
    Let $\eta \in \G_\sfP$ with $q-1 > \max_{\g \in \Psi_\sfN} \eta(\g)$, then $[\G_\eta^F, \sfN^F] =  \sfN^F$.
\end{lemma}
\begin{proof}
    For $\g \in \Psi_\sfL$ we denote by $\sfL^\g$ the corresponding root subgroup. By assumption on $\eta$, there exists $z \in \eta(\BF_q) \subseteq \G_\eta^F$ such that
    \[\tag{a} \text{ the map $x \mapsto [z, x]$ gives an automorphism of $\sfL^\g$ for $\g \in \Psi_\sfN$.}\]

    For $i \in \BZ_{\ge 1}$ let $\sfN_i$ be the normal subgroup of $\sfN$ generated by the root subgroups $\sfL^\g$ with $\eta(\g) \ge i$. Fix a representative set $\{\g_1^i, \dots, \g_{n_i}^i\}$ for the $F$-orbits of the subset $\{\g \in \Psi_\sfN; \eta(\g) = i\}$. Then the natural projection map $\sfN_i \to \prod_{j=1}^{n_i} \sfL^{\g_j^i}$ gives an isomorphism of $\BF_q$-linear spaces \[\tag{b} \sfN_i^F / \sfN_{i+1}^F \cong (\sfN_i/ \sfN_{i+1})^F \cong \bigoplus_{j=1}^{n_i} (\sfL^{\g_j^i})^{F^{m_j}},\] where $m_j$ denotes the order of the $F$-orbit of $\g_j^i$.

    By (a) and (b), the map $x \mapsto [z, x]$ gives an automorphism of $\sfN_i^F / \sfN_{i+1}^F$ for $i \in \BZ_{\ge 1}$. Thus $\sfN^F$ is generated by the elements $[z, x]$ for $x \in \sfN^F$, and hence $[\G_\eta^F, \sfN^F] = \sfN^F$ as desired.
\end{proof}

Now we are ready to show the cuspidality of $\rho_0^\flat$. The proof is inspired from the proof of \cite[Theorem 3.1]{F21a}
\begin{proof}[Proof of Proposition \ref{cuspidal}]
    If $L = G$, then $\k(\L) \cong \k_\phi \cong \phi_0$ and hence $\rho_0^\flat$ is isomorphic to the cuspidal representation $\rho_0$ as desired.

    Assume $L \subsetneq G$. Since $\bar\tx$ is vertex of $\CB(L_\der, k)$, we have $r_{ss}(G) \ge r_{ss}(L) + 1 = r_{ss}(\sfL) + 1$. Suppose $\rho_0^\flat$ is not cuspidal. Then up to conjugation by $\sfL^F$, there exists a standard maximal proper parabolic subgroup $\sfB \subseteq \sfP \subsetneq \sfL = L_0$ with unipotent radical $\sfN$ such that $\rho_0^\flat |_{\sfN^F}$ contains the trivial $\sfN^F$-module. Let $\eta \in \L_\sfP$ such that $q-1 > \max_{\g \in \Psi_\sfN} \eta(\g)$. Due to Lemma \ref{max}, such $\eta$ exists if $q > r_{ss}(G) + 1$. By Lemma \ref{radical} we have $\sfL^F = [\G_\eta^F, \G_\eta^F]$.

    Let $\Phi_\eta^\pm$ and $\Phi_\eta^0$ be the set of roots $\g \in \Phi(G, T)$ such that $\pm \eta(\g) > 0$ and $\eta(\g) = 0$ respectively. Let $U_\eta^\pm = \prod_{\a \in \Phi_\eta^\pm} G^\a$ and $M_\eta \subseteq G$ the Levi subgroup generated by $T$ and $G^\a$ for $\a \in \Phi_\eta^0$.

    Let $\CH = H^F / \ker\phi^\natural$ be the Heisenberg $p$-group, whose center is denoted by $C = (K^+)^F / \ker\phi^\natural$. Let $V = \CH / C = H^F / (K^+)^F$ be the symplectic quotient. Then \[V = V_+ \oplus V_0 \oplus V_-,\] where $V_\pm$ and $V_0$ are the natural images of $(H \cap U_\eta^\pm)^F$ and $(H \cap M_\eta)^F$ in $V$ respectively. Similarly, let $\CH_\pm$ and $\CH_0$ be the natural images of $(H \cap U_\eta^\pm)^F$ and $(H \cap M_\eta)^F$ in $\CH$ respectively. Then $\CH_0$ is the inverse image of $V_0$ under the natural projection $\pi: \CH \to V$, which is also a Heisenberg $p$-group. We fix a totally isotropic subspace $\CL_0 \subseteq \CH_0$ such that $\pi |_{\CL_0}$ is a bijection from $\CL_0$ to a Lagrangian subspace of $V_0$. We set \[\CL = \CH_+ \oplus \CL_0.\] Then $\CL \subseteq \CH$ is a totally isotropic subspace such that $\pi |_\CL$ is a bijection from $\CL$ to a Lagrangian subspace of $V$. Moreover, by definition we have $y\i \Ad_z(y) \in \CH_+$ for any $z \in \G_\eta^F$ and $y \in \CH_+ \oplus \CH_0$. Here $\Ad_z: \CH \to \CH$ denotes the natural conjugation action of $z \in \G_\eta^F$ on $\CH$. In particular, $\CL$ is normalized by $\G_\eta^F$. Thus $\G_\eta^F$ acts on the $\CL$-invariant subspace \[ \k(\L)^\CL \otimes \rho_0 = (\k(\L) \otimes \rho_0)^\CL \cong (\k_\phi \otimes \rho_0^\flat)^\CL = (\k_\phi)^\CL \otimes \rho_0^\flat.\] As $\k(\L) |_\CH \cong \k_\phi |_\CH$ are the same irreducible Heisenberg representations of $\CH$, $\k(\L)^\CL$ and $(\k_\phi)^\CL$ are both one-dimensional. Thus $\sfN^F = [\G_\eta^F, \G_\eta^F]$ acts trivially on $\k(\L)^\CL$ and $(\k_\phi)^\CL$. Then we have \[((\k_\phi)^\CL \otimes \rho_0^\flat)^{\sfN^F} = (\k_\phi)^\CL \otimes (\rho_0^\flat)^{\sfN^F} \neq \{0\}.\] On the other hand, since $\rho_0$ is cuspidal, we have \[(\k(\L)^\CL \otimes \rho_0)^{\sfN^F} = \k(\L)^\CL \otimes (\rho_0)^{\sfN^F} = \{0\}.\] This is a contradiction and the proof is finished.
\end{proof}

\subsection{Realization of supercuspidal representations} By Proposition \ref{cuspidal}, \cite[Corollary 7.7 \& Proposition 8.2]{DL} and the assumption that $\tx \in \CB(L, k)$ is a vertex, there exist an unramified elliptic maximal torus $S$ of $L$ with $\tx \in \CA(S, \brk)$ and a character $\l$ of $S_\tx^F$ of depth zero such that $\rho_0^\flat |_{L_\tx^F}$ is inflated from a direct summand of $R_{S, 0}^L(\l)$. In particular, $\l |_{Z_\tx^F} = \rho_0^\flat |_{Z_\tx^F} = \o |_{Z_\tx^F}$. Moreover, since $Z(L) / Z(G)$ is anisotropic, $S$ is also an elliptic maximal torus of $G$. As $S$ is unramified and elliptic, we have $S^F = Z^F S_\tx^F$. Then we can extend $\l$ to a character of $S^F$ whose restriction to $Z^F$ is $\o$. We still denote it by $\l$. We extend the $L_\tx^F$-module $R_{S, 0}^L(\l)$ to a $Z^F K^F$-module such that the action of $z h$ on $R_{S, 0}^L(\l)$ for $z \in Z^F$ and $h \in H^F$ is given by $\l(z)$.

Let $\psi = \l \prod_{i=0}^d \phi_i|_{S^F}$ be a character of $S^F$. Let $g \in L_\tx$ such that $g T g\i = S$. Set $V = g U g\i$. We view $R_{S, V, r_\psi}^G(\psi)$ as a representation of $Z^F G_\tx^F = S^F G_\tx^F$ such $Z^F$ acts via the character $\psi$.

Since $g \in L = G^0$ and $(G^i)_{0 \le i \le d}$ is standard with respect to $U$, $(G^i)_{0 \le i \le d}$ is also standard with respect to $V$. Let $\k_\psi = \k_{\psi, V}$ be the $K^F$-module constructed in \S\ref{sec-decomp}.
\begin{lemma} \label{S-T}
    We have $\k_{\psi, V} \cong \k_{\phi, U}$ as $K^F$-modules.
\end{lemma}
\begin{proof}
    As $(\L, 1)$ and $(\L, \l)$ are Howe factorizations of $\phi$ and $\psi$ respectively, by definition we have $\phi^\natural = \psi^\natural = \chi_\L = \chi$, $\CI_{\psi, V} = g \CI_{\phi, U} g\i$, $H_\psi = H_\phi = H$ and $\CI_{\psi, V} = g \CI_{\phi, U} g\i$. Since $(G^i)_{0 \le i \le d}$ is standard with respect to $U$ we have \[\CI_{\psi, V} \cap H_\psi = g (\CI_{\phi, U} \cap H_\phi) g\i = \CI_{\phi, U} \cap H_\phi.\] Therefore, $\bar Z_{\psi, V, r}^H = \bar Z_{\phi, U, r}^H$ for $r \ge r_\phi = r_\psi$. Note that $(K^+)^F \subseteq H^F$ also acts on $\bar Z_{\psi, V, r}^H$ by right multiplication. It follows from Lemma \ref{coincidence} that \[\k_{\psi, V} =  H_c^*(\bar Z_{\psi, V, r}^H, \ov\BQ_\ell)[\chi] = (\bar Z_{\phi, U, r}^H, \ov\BQ_\ell)[\chi] = \k_{\phi, U}.\] The proof is finished.
\end{proof}

The main result of this section is as follows.
\begin{theorem} \label{sum}
    Assume that $q > r_{ss}(G) + 1$. The supercuspidal representation $\pi_\Xi$ is a direct summand of $\text{c-}\ind_{Z^F G_\tx^F}^{G^F} R_{S, V, r_\psi}^G(\psi)$ with $r_\psi$ the depth of $\psi$.
\end{theorem}
\begin{proof}
Note that $(\L, \l)$ is a Howe factorization of $\psi$. We have \begin{align*}
    \pi_\Xi &\cong \text{c-}\ind_{\tilde K^F}^{G^F} \k(\L) \otimes \rho \\ &\subseteq \text{c-}\ind_{Z^F K^F}^{G^F} \k(\L) \otimes \rho_0 \\ &\cong \text{c-}\ind_{Z^F K^F}^{G^F} \k_\phi \otimes \rho_0^\flat \\  &\cong \text{c-}\ind_{Z^F K^F}^{G^F} \k_\psi \otimes \rho_0^\flat \\ &\subseteq \text{c-}\ind_{Z^F K^F}^{G^F} \k_\psi \otimes R_{S, 0}^L(\l) \\ &\cong \text{c-}\ind_{Z^F K^F}^{G^F} \k_\psi \otimes R_{S, 0}^L(\l) \\ &\cong \text{c-}\ind_{Z^F G_\tx^F}^{G^F} R_{S, V,r_\psi}^G(\psi),
\end{align*} where the second isomorphism follows from Lemma \ref{iso}, the third one follows from  Lemma \ref{S-T}, and the last one follows from Theorem \ref{decomposition} together with the natural bijection $G_\tx^F / K^F \cong Z^F G_\tx^F / (Z^F K^F)$.
\end{proof}

\end{document}